\documentclass[a4paper]{amsart}
\usepackage{amsmath,amssymb,amsthm,enumerate}
\title[Volume renormalization for the Blaschke metric]
{{\bf Volume renormalization for the Blaschke metric on strictly convex domains}}
      
\author{TAIJI MARUGAME}
\date{}

\newcommand\R{\mathbb{R}}

\newcommand\Ric{{\rm Ric}}
\newcommand\scal{{\rm Scal}}
\renewcommand\a{\alpha}
\renewcommand\b{\beta}
\newcommand\g{\gamma}
\renewcommand\d{\delta}
\newcommand\e{\epsilon}
\renewcommand\r{\rho}
\renewcommand\th{\theta}
\newcommand\br{\boldsymbol{\rho}}
\newcommand\bu{\boldsymbol{u}}
\newcommand\bh{\boldsymbol{h}}
\newcommand\bS{\boldsymbol{S}}
\newcommand\bxi{\boldsymbol{\xi}}
\newcommand\bnabla{\boldsymbol{\nabla}}
\newcommand\U{\Upsilon}
\newcommand\lam{\lambda}
\newcommand\pa{\partial}
\newcommand{\calE}{\mathcal{E}}
\newcommand{\calJ}{\mathcal{J}}
\newcommand{\calO}{\mathcal{O}}
\newcommand{\wt}{\widetilde}
\newcommand{\wh}{\widehat}

\newtheorem{thm}{Theorem}[section]
\newtheorem{dfn}[thm]{Definition}
\newtheorem{prop}[thm]{Proposition}
\newtheorem{rem}[thm]{Remark}
\newtheorem{lem}[thm]{Lemma}

\makeatletter
\@addtoreset{equation}{section}

\makeatother
\address{Graduate School of Mathematical Sciences, The University of Tokyo,
	3-8-1 Komaba, Meguro, Tokyo 153-8914, Japan}
\email{marugame@ms.u-tokyo.ac.jp}

\begin{document}
\maketitle
\begin{abstract} 
We consider the volume expansion of the Blaschke metric, which is a projectively invariant metric on a strictly convex domain in a locally flat projective manifold. When the boundary is even dimensional, we express the logarithmic coefficient $L$ as the integral of affine invariants over the boundary. We also formulate an intrinsic geometry of the boundary as a conformal Codazzi structure and show that $L$ gives a global conformal invariant of the boundary. 
\end{abstract}
\section{Introducton}
Renormalization of the volume of a complete metric has been an important topic in conformal and CR 
geometries. In conformal geometry, the complete metric we consider is the Poincar\'e-Einstein metric 
(\cite{G}, \cite{HS}), which generalizes the hyperbolic metric on the unit ball. The volume renormalization of Poincar\'e-Einstein metric was motivated by theoretical physics and it has been intensively studied also in mathematical viewpoints (\cite{Al}, \cite{An}, \cite{GMS}). A Riemannian metric $g_+$ on the interior $X$ of a compact manifold with boundary is called a Poincar\'e-Einstein metric if it satisfies the Einstein equation and $x^2g_+$ extends to a metric on $\overline{X}$, where $x$ is a defining function of $X$ which is positive inside. The conformal class of $x^2g_+|_{T\partial X}$ is independent of the choice of $x$ and defines a conformal structure $[h]$ on $M=\partial X$, which is 
called the conformal infinity of $g_+$. The volume of $X$ with respect to $g_+$ is infinite, but one can renormalize it by expanding the volume of  subdomains; a representative metric $h\in[h]$ determines a special choice of a 
defining function $x$ with $x^2g_+|_{TM}=h$ and the volume of each sublevel set of $x$ has the following expansion:
\begin{equation}\label{vol-exp-PE}
{\rm Vol}(\{x>\e\}) \\
=\sum_{j=0}^{\lceil n/2\rceil-1}c_{-n+2j}\e^{-n+2j}+
\begin{cases}V+o(1) & (n: {\rm odd}) \\
L\log\frac{1}{\e}+V+o(1) & (n: {\rm even}),
\end{cases}
\end{equation}
where $n=\dim M$. The constant term $V$, called the {\it renormalized volume}, is independent of the choice of $h\in[h]$ when $n$ is odd. When $n$ is even $V$ depends on the choice, but instead the coefficient $L$ of the logarithmic term becomes an invariant of $g_+$. Moreover, it is shown in \cite{GZ}, \cite{FG} that $L$  agrees with the integral of the $Q$-curvature over $M$, which is a global conformal invariant. 

The complex counterpart of this setting is the Cheng--Yau metric on a strictly pseudoconvex domain with CR boundary in a complex manifold. The Cheng--Yau metric is a complete K\"ahler-Einstein metric constructed from the solution to the complex Monge--Amp\`ere equation (\cite{CY2}). In this case, there are two ways of choosing a defining function for the volume renormalization. Seshadri \cite{Se} uses a defining function which satisfies a normalization condition analogous to that in the conformal case and shows that the coefficient of the logarithmic term agrees with the total CR $Q$-curvature, though it vanishes for a natural class of contact forms, namely the class of pseudo-Einstein contact forms. On the other hand, the volume expansion with respect to the approximate solution to the Monge--Amp\`ere equation is considered in \cite{HMM}. The expansion has no logarithmic term but the constant term is shown to agree, up to a topological term, with the total $Q$-prime curvature, which is a global CR invariant defined by Case--Yang \cite{CaY} and Hirachi \cite{H} for pseudo-Einstein contact forms.

Conformal geometry and CR geometry are both examples of so called {\it parabolic geometries}. A parabolic geometry is a Cartan geometry modeled on the homogeneous space $G/P$, where $G$ is a semi-simple Lie group and $P$ is a parabolic subgroup, and can be treated in a unified way by Cartan connections, tractor calculus, and BGG machinery (\cite{CG1}, \cite{CS}, \cite{CSS}). 

In this article, we consider the volume renormalization in the setting of another important example of such geometries: {\it projective differential geometry}. Recall that two torsion-free affine connections are 
said to be projectively equivalent if they have the same geodesics up to parametrization. A differentiable manifold $N$ equipped with a projective equivalence class $[\nabla]$ (a projective structure) is called a projective manifold. The flat model is the projective space $\R\mathbb{P}^{n+1}$ with the canonical projective structure and we say $(N, [\nabla])$ is locally flat if it is locally projectively equivalent to the flat model. We view a locally flat projective structure as a real analogue of complex structure, and consider a strictly convex domain $\Omega$ in a projective manifold instead of a strictly pseudocovex domain. Along the analogy to the complex case, we define the real  Monge--Amp\`ere 
equation for a projective density $\br\in\calE(2)$ as 
\begin{equation}\label{MA-intro}
\det D_ID_J\br=-1, \quad \Omega=\{\br<0\}
\end{equation}
with the $D$-operator on the tractor bundle. (See \S2 for the definitions of $\calE(2)$ and $D_I$.) Then, with the (approximate) solution to \eqref{MA-intro}, we define a projectively invariant metric $g$ on $\Omega$ by
$$
g_{ij}=\frac{\nabla_i\nabla_j\br}{-2\br}+\frac{\nabla_i\br\nabla_j\br}{4\br^2}-P_{ij},
$$
where $P_{ij}$ is the projective Schouten tensor of $\nabla\in[\nabla]$. We call $g$ the {\it Blaschke metric} by following Sasaki \cite{Sa1}. The Blaschke metric is a generalization of the Klein (projective) 
model hyperbolic metric while the Poincar\'e-Einstein metric is a generalization of the Poincar\'e model hyperbolic metric. Fixing a projective scale $\tau\in\calE(1)$, we obtain the same volume expansion as 
in \eqref{vol-exp-PE} with $x=(-2\tau^{-2}\br)^{1/2}$, and $V$ (resp. $L$) is independent of the choice of 
$\tau$ when $n$ is odd (resp. even), where $n$ is the dimension of $M=\partial\Omega$ (Theorem
\ref{volume}). We then relate the invariant $L$ to the geometry of the boundary. If we fix a scale $\tau$, 
the boundary $M$ is endowed with several affine tensors such as the affine metric $h_{\a\b}$, the affine shape operator $S_\a{}^\b$, and the Fubini--Pick form $A_{\a\b\g}$. Under a rescaling $\wh\tau=e^{-\U}\tau$ with $\U\in C^{\infty}(N)$, $h_{\a\b}$ and $A_{\a\b\g}$ transform conformally and the transformation formula of $S_\a{}^\b$ involves the 1-jet of $\U$ along $M$. Given a 1-jet of a projective scale along the boundary, we introduce a normalized scale via the tractor parallel transport so that $L$ is written in terms of the integral of affine invariants (Theorem \ref{L-exp}).

Although the shape operator $S_\a{}^\b$ is an extrinsic curvature of $M$, there is an intrinsic definition 
of the geometric structure on a strictly convex hypersurface in a locally flat projective manifold: A {\it conformal Codazzi structure} is a conformal structure $[h]$ (or a M\"obius structure when $n=2$) together with a trace-free symmetric tensor $A_{\a\b\g}\in\calE_{(\a\b\g)_0}[2]$ which satisfies certain integrability conditions called the {\it Gauss--Codazzi equations} (see Definition \ref{conf-Codazzi}). This is a special case of the framework of {\it parabolic subgeometries} proposed by Burstall--Calderbank \cite{BC}. They develop the general theory of parabolic geometry on a submanifold of a generalized flag manifold, and derive the Gauss--Codazzi equations in terms of Lie algebra homology and the BGG operators. In this paper, we derive an explicit form of the equations in a rather direct way based on the comparison of the ambient projective tractor bundle and the conformal tractor bundle over the boundary;
these two bundles are canonically isomorphic and the projective tractor connection induces a connection on the conformal tractor bundle over $M$, whose flatness corresponds to the Gauss--Codazzi equations. The form of the equations depends on whether $n=2, 3$ or $\ge4$. When $n\ge4$, the equations are 
\begin{align*}
&({\rm Gauss}) &
 2\,{\rm tf}(A_{\nu\g[\a}A_{\b]\mu}{}^\nu)+W^h_{\a\b\g\mu}&=0,  \\
&({\rm Codazzi}) &
\nabla^h_{[\a}A_{\b]\g\mu}-\frac{1}{n}\Bigl(h_{\mu[\a}(\d A)_{\b]\g}+h_{\g[\a}(\d A)_{\b]\mu}\Bigr)&=0,
\end{align*}
where ${\rm tf}$ denotes the trace-free part, $W^h_{\a\b\g\mu}$ is the Weyl tensor, and $(\d A)_{\a\b}:=\nabla^h_\g A_{\a\b}{}^\g$.
See Proposition \ref{GC-prop} for the Gauss--Codazzi equations in $n=2,3$. The differential operator applied to $A_{\a\b\g}$ in the Codazzi equation is known as the second BGG operator associated with the representation $S^2_0 (\R^{n+2})^\ast$ of $SO(n+1, 1)$.  
Conversely, given a tensor $A_{\a\b\g}$ on a conformal manifold with the Gauss--Codazzi equations, one can define a local immersion to the projective space such that the induced conformal Codazzi structure agrees with the original one (Theorem \ref{proj bonnet}, the projective Bonnet theorem). Thus it contains complete information on the local geometric structure on a hypersurface in a projective manifold.

In order to relate the invariant $L$ to the conformal Codazzi structure on the boundary, we introduce a normalization on the 1-jet of the projective scale by imposing the  harmonicity $\wt\Delta\tau=O(\br)$ with respect the ambient metric $\wt g_{IJ}=D_ID_J\br$. Any conformal scale on the boundary extends to such a scale and we call it a {\it harmonic scale}. In  a harmonic scale, the shape operator is expressed with the Fubini--Pick form and we obtain a representation of $L$ as a conformal invariant on a conformal Codazzi manifold (Theorem \ref{L-codazzi}). Explicit formulas are given for $n=2, 4$ (see \eqref{L2}, \eqref{L4}). 

Using the ambient metric, one can construct GJMS operators and $Q$-curvature for conformal Codazzi geometry and show that the total $Q$-curvature agrees with $L$ as in conformal and CR cases; the details including some variation formulas will appear in \cite{M}. 
\medskip

This paper is organized as follows. In \S2, we review basic notions in projective differential geometry including projective tractor calculus. In \S3, we define the Blaschke metric via the solution to the real 
Monge--Amp\`ere equation and consider the volume expansion to obtain invariants of the domain. 
In \S4, we express the logarithmic coefficient $L$ as the integral of affine invariants over the boundary. 
We then derive the explicit form of the Gauss--Codazzi equations in \S5 and define the conformal Codazzi structures. We also introduce harmonic scales in order to represent $L$ as a conformal invariant. 
\medskip

\noindent
{\it Notations.}
We adopt Einstein's summation convention and assume that 
\begin{itemize}
\item
uppercase Latin indices $I, J, K, \dots$ run from 0 to $n+1$;
\item
lowercase Latin indices
$i, j, k, \dots$
run from 1 to $n+1$; 
\item
lowercase Greek indices $\alpha,\beta,\gamma,\dots$ run from 1 to $n$. 
\end{itemize}
\medskip
\noindent{\bf Acknowledgments.} This paper is based on part of the author's thesis at the University of Tokyo. I would like to express my deep gratitude to my advisor Professor Kengo Hirachi for his support and encouragement throughout this work. I would also 
like to thank Dr. Yoshihiko Matsumoto, Professor Michael Eastwood, and Professor 
Bent \O rsted for invaluable comments on the results. This research was partially supported by JSPS Fellowship and
 KAKENHI 13J06630.

\section{Projective structure and tractor bundle}
\subsection{Projective structures}Let $N$ be an oriented $C^{\infty}$-manifold of dimension $n+1\ge 3$. We use abstract index notation to denote tensors and tensor bundles over $N$. For example, $\calE^i$ and $\calE_i$ stand for $TN$ and $T^\ast N$ respectively, and sections  of the bundles are denoted by $X^i$ or $T_i$. The symmetrization and skew symmetrization of a tensor $T_{i_1i_2\dots i_k}$ are denoted by $T_{(i_1i_2\dots i_k)}$ and $T_{[i_1i_2\dots i_k]}$ respectively. The corresponding tensor bundles are denoted by $\calE_{(i_1i_2\dots i_k)}$ or $\calE_{[i_1i_2\dots i_k]}$, and we also use the same symbols to denote the space of sections of the bundles. Since $N$ is oriented we can define, for each $w\in\R$, an oriented real line bundle $\calE(w)=(\wedge^{n+1} T^\ast N)^{-w/(n+2)}$, which we call the {\it projective density bundle of weight $w$}.  A connection $\nabla$ on $TN$ induces a connection on $\calE(w)$, which we also denote by $\nabla$. A positive section of $\calE(1)$ is called a {\it projective scale}. We put $(w)$ to a bundle to indicate the tensor product with $\calE(w)$, as in $\calE_{ij}(w)$.

Let $\nabla$ and $\nabla^\prime$ be torsion-free affine connections on $TN$. We say $\nabla$ and $\nabla^\prime$ are {\it projectively equivalent} when there exists a 1-form $p_i$ such that 
$$
\nabla^\prime_i X^j=\nabla_i X^j+p_iX^j+p_lX^l \d_i{}^j,
$$
where $\d_i{}^j$ is the Kronecker delta. We write this relation as $\nabla^\prime=\nabla+p$. It is known that two affine connections are projectively equivalent if and only if they have the same unparametrized geodesic paths; see, e.g., \cite{CS} for a proof. An equivalence class $[\nabla]$ is called a {\it projective structure} on $N$. The curvature and Ricci tensors of $\nabla\in [\nabla]$ are defined by 
$$
(\nabla_i\nabla_j-\nabla_j\nabla_i)X^k=R_{ij}{}^k{}_lX^l, \quad \Ric_{jl}=R_{ij}{}^i{}_l.
$$
Note that the Ricci tensor is not symmetric in general. To deal with this matter, we introduce a subclass of 
$[\nabla]$, which consists of Ricci symmetric connections.
\begin{lem}
For each choice of a projective scale $\tau\in\calE(1)$, there exists a unique $\nabla\in[\nabla]$ such that $\nabla\tau=0$. Moreover the Ricci tensor of $\nabla$ is symmetric.
\end{lem}
\begin{proof}
Let $\mathring{\nabla}$ be an arbitrary representative connection. Take a local unimodular frame 
$\{e_i\}$ with respect to the volume $\tau^{-(n+2)}$ and let $\mathring{\omega}_i{}^j$ be the connection forms of $\mathring{\nabla}$. Then for $\nabla=\mathring{\nabla}+p\in[\nabla]$, we have
$$
\nabla\tau=\bigl((n+2)^{-1}\mathring{\omega}_l{}^l+p\bigr)\otimes\tau.
$$
Thus there is a unique $p$ for which $\nabla\tau=0$. The Bianchi identity $R_{ij}{}^k{}_k=2\,\Ric_{[ji]}$ implies that $\nabla$ is Ricci symmetric.
\end{proof}

From now on, we will restrict ourselves to representative connections given as in the above lemma. If we change a projective scale as $\wh{\tau}=e^{-\U}\tau, \U\in C^\infty (N)$, the connection transforms by $\wh{\nabla}=\nabla+d\U$. We define the {\it projective Schouten tensor} $P_{ij}\in\calE_{(ij)}$ and the {\it projective Weyl tensor} $C_{ij}{}^k{}_l\in\calE_{[ij]}{}^k{}_l$ by
$$
P_{ij}=\frac{1}{n}\Ric_{ij},\quad C_{ij}{}^k{}_l=R_{ij}{}^k{}_l-2\,\d_{[i}{}^kP_{j]l}.
$$ 
For a change of the scale $\wh{\tau}=e^{-\U}\tau$, the projective Weyl tensor is invariant and the projective Schouten tensor satisfies
$$
\wh{P}_{ij}=P_{ij}-\nabla_i\U_j+\U_i\U_j,
$$
where $\U_i=d\U$. A projective structure is said to be {\it locally flat} if $C_{ij}{}^k{}_l=0$. This is equivalent to the condition that on a neighborhood of any point in $N$ there exists a projective scale such that the corresponding connection in the projective class is flat.
We call such a projective scale an {\it affine scale}.
From the Bianchi identity $\nabla_{[i}R_{jk]}{}^l{}_m=0$ it follows that 
$\nabla_k C_{ij}{}^k{}_l=2(n-1)\nabla_{[i}P_{j]l}$, so if $[\nabla]$ is locally flat $\nabla_{[i}P_{j]l}$ also vanishes. 

\subsection{Projective tractor bundle and tractor connection}
We review the construction of the tractor bundles and tractor connections for projective manifolds by following \cite{BEG}. We refer to \cite{CG1} for the general construction of tractor bundles for parabolic geometries.

Let $\pi: \wt N\rightarrow N$ be the $\R_+$-bundle of positive elements in $\calE(-1)$. Let $\R_+$ act on $T\wt N$ by $s\cdot u:=s^{-1}(\d_s)_\ast u$, where $\d_s$ is the dilation in the fibers of $\wt N$.
Then $\wt \calE^I:=T\wt N/\R_+$ becomes a rank $(n+2)$ vector bundle over $N$, and is called the 
{\it projective tractor bundle}. We call an element $\nu^I\in\wt\calE^I$ a {\it tractor}. We often identify a section of $\wt\calE^I$ with a vector field on $\wt N$ homogeneous of degree $-1$. Since the inclusion $\pi^\ast\calE(-1)\subset T\wt N$ is $\R_+$-invariant, $\wt \calE^I$ contains $\calE(-1)$ as a line subbundle, which is the kernel of the well-defined projection $\wt \calE^I\longrightarrow \calE^i(-1)$. Thus we have an exact sequence 
\begin{equation}\label{Euler}
0\longrightarrow \calE(-1)\longrightarrow \wt \calE^I \longrightarrow \calE^i(-1)\longrightarrow 0.
\end{equation}
If we choose a projective scale $\tau\in\calE(1)$ then $\wt N$ is trivialized by $\tau^{-1}$ and accordingly the sequence 
\eqref{Euler} splits: 
$$
\wt\calE^I\overset{\tau}{\cong}
\begin{matrix}
\calE^i(-1) \\
\oplus \\
\calE(-1)
\end{matrix}.
$$
So a tractor is represented as $\nu^I={}^t(\nu^i, \lam)$ with $\nu^i\in\calE^i(-1)$ and 
$\lam\in\calE(-1)$. Let $\nu^I={}^t(\wh\nu^i, \wh\lam)$ be the expression in terms of another scale $\wh\tau=e^{-\U}\tau$. Then two expressions are related as
\begin{equation}\label{tract-trans}
\begin{pmatrix}
\wh\nu^i \\
\wh\lam
\end{pmatrix}
=
\begin{pmatrix}
\nu^i \\
\lam-\U_k\nu^k
\end{pmatrix}.
\end{equation}
The inclusion $\calE(-1)\hookrightarrow\wt\calE^I$ is described as $\lam\mapsto\lam T^I$, where 
$T^I:={}^t(0,1)\in\wt\calE^I(1)$ is the {\it Euler field}.

The {\it tractor connection} is a projectively invariant linear connection on $\wt\calE^I$ defined by 
\begin{equation}\label{tract-conn}
\nabla_i
\begin{pmatrix}
\nu^j \\
\lam
\end{pmatrix}
=
\begin{pmatrix}
\nabla_i\nu^j+\lam\d_i{}^j \\
\nabla_i\lam-P_{ik}\nu^k
\end{pmatrix}
\end{equation}
in a chosen projective scale. We can also consider the dual or tensor products of the projective tractor bundles, and we also call them ``tractor bundles''. The transformation laws and the induced tractor connections are computed from \eqref{tract-trans} and \eqref{tract-conn}. For the dual bundle $\wt\calE_I\overset{\tau}{\cong}\calE(1)\oplus\calE_i(1)$, we have
$$
\begin{pmatrix}
\wh\sigma\\
\wh\mu_i
\end{pmatrix}
=
\begin{pmatrix}
\sigma \\
\mu_i+\U_i\sigma
\end{pmatrix},
\quad
\nabla_i
\begin{pmatrix}
\sigma \\
\mu_j
\end{pmatrix}
=
\begin{pmatrix}
\nabla_i\sigma-\mu_i \\
\nabla_i\mu_j+P_{ij}\sigma
\end{pmatrix}.
$$
We define the curvature of the tractor connection by 
$$
(\nabla_i\nabla_j-\nabla_j\nabla_i)\nu^K=\Omega_{ij}{}^K{}_L\nu^L.
$$
Then we have
$$
\Omega_{ij}{}^K{}_L=
\begin{pmatrix}
C_{ij}{}^k{}_l & 0 \\
-2\nabla_{[i}P_{j]l} & 0
\end{pmatrix}.
$$
Thus the projective tractor connection is flat if and only if the projective structure $[\nabla]$ is 
locally flat.

Let $\wt\calE_\ast(w)$ denote a weighted tractor bundle with arbitrary tractor indices. 
The {\it projective tractor $D$-operator} $D_I : \wt\calE_\ast(w)\rightarrow\wt\calE_{I\ast}(w-1)$ 
is defined by 
$$
D_If_\ast=
\begin{pmatrix}
wf_\ast \\
\nabla_i f_\ast
\end{pmatrix}
\in\, \begin{matrix}
\wt\calE_\ast(w) \\
\oplus \\
 \wt\calE_{i\ast}(w)
\end{matrix}
\overset{\tau}{\cong}
 \wt\calE_{I\ast}(w-1),
$$
where in the second slot, we use the connection induced by the tractor connection and the flat connection on $\calE(w)$ relative to a chosen scale $\tau\in \calE(1)$. It is immediate to check that $D_I$ is a projectively invariant differential operator. A calculation shows that the commutator of  the $D$-operator satisfies $T^ID_{[I}D_{J]}f_\ast=0$ and it projects to $\nabla_{[i}\nabla_{j]}f_\ast$. Therefore 
$D_{[I}D_{J]}f_\ast=0$ for locally flat projective structures. 

We can give a characterization of the affine scale via projectively invariant differential operators:
\begin{prop}\label{affine-scale}
Assume that $[\nabla]$ is locally flat and let $\tau\in\calE(1)$ be a projective scale.  Then the following are equivalent:
\begin{itemize}
\item[(i)] The scale $\tau$ is affine;
\item[(ii)] It holds that $D_ID_J\tau=0$;
\item[(iii)] It holds that $\nabla_i\nabla_j\tau+P_{ij}\tau=0$ for any $\nabla\in[\nabla]$.
\end{itemize}
\end{prop}
\begin{proof}
The equivalence of (ii) and (iii) follows from the equation
$$
D_ID_J\tau=
\begin{pmatrix}
0 & 0 \\
0 & \nabla_i\nabla_j\tau+P_{ij}\tau
\end{pmatrix}.
$$
If we compute in the scale $\tau$, we have $\nabla_i\nabla_j\tau+P_{ij}\tau=P^\tau_{ij}\tau$, where $P^\tau_{ij}$ is the projective Schouten tensor of the corresponding representative connection $\nabla\in[\nabla]$. Thus $\nabla_i\nabla_j\tau+P_{ij}\tau=0$ if and only if $P^\tau_{ij}=0$, which is equivalent to the flatness of $\nabla$ by $C_{ij}{}^k{}_l=0$.
\end{proof}

\section{Monge--Amp\`{e}re equation and Blaschke metric}
\subsection{Monge--Amp\`{e}re equation on strictly convex domains}\label{MA-convex}
Let $(N, [\nabla])$ be an oriented $(n+1)$-dimensional projective manifold. Let $\Omega$ be a relatively 
compact domain in $N$ with smooth boundary $\pa\Omega=M$. We say $M$ is {\it strictly convex} if the Hessian $(\nabla_i\nabla_j\r)|_{TM}$ is positive definite for a connection $\nabla\in[\nabla]$ and a defining function $\r$. (Here and after, the defining functions or defining densities are assumed to be {\it negative} inside the domain.) Note that the condition is independent of the choices of $\nabla$ and $\r$.

In order to formulate the real Monge--Amp\`{e}re equation on $\Omega$, we will define the determinant of an element in $\wt\calE_{ij}(w)$ or $\wt\calE_{IJ}(w)$. Take a projective scale $\tau\in\calE(1)$. First we define 
$$
\det:\wt\calE_{ij}(w)\longrightarrow\calE\bigl((n+1)w-2n-4\bigr)
$$
by
$$
\det \phi_{ij}:=\tau^{(n+1)w-2n-4}\det{}_\tau(\tau^{-w}\phi_{ij}),
$$
where in the right-hand side we take the determinant in a unimodular frame with respect to the volume $\tau^{-(n+2)}$. Then, $\det \phi_{ij}$ is independent of $\tau$ since we have $\det{}_{\wh\tau}=e^{-(2n+4)\U}\det{}_\tau$ for another projective scale $\wh\tau=e^{-\U}\tau$.
Next we define 
$$
\det:\wt\calE_{IJ}(w)\longrightarrow\calE\bigl((n+2)w\bigr)
$$
by
$$
\det
\begin{pmatrix}
\sigma & \nu_j \\
\mu_i   & \lam_{ij}
\end{pmatrix}
:=\tau^{(n+2)w}\det{}_\tau\Bigl[\tau^{-w-2}
\begin{pmatrix}
\sigma & \nu_j \\
\mu_i   & \lam_{ij}
\end{pmatrix}
\Bigr].
$$
Note that each component in the matrix has weight $w+2$. One can show that the determinant  is invariant under the change $\wh\tau=e^{-\U}\tau$ by using the transformation law
$$
\begin{pmatrix}
\wh\sigma & \wh\nu_j \\
\wh\mu_i   & \wh\lam_{ij}
\end{pmatrix}
=
\begin{pmatrix}
\sigma & \nu_j+\U_j\sigma \\
\mu_i+\U_i\sigma   & \lam_{ij}+\U_j\mu_i+\U_i\nu_j+\U_i\U_j\sigma
\end{pmatrix}.
$$
Now the real Monge--Amp\`{e}re equation for a density $\br\in\calE(2)$ is defined by
\begin{eqnarray}\label{MA1}
\begin{cases}
\calJ[\br]=-1, \\
\Omega=\{\br<0\},
\end{cases}
\end{eqnarray}
where
$$
\calJ[\br]:=\det(D_ID_J\br) =\det\begin{pmatrix} 2\br & \nabla_j\br \\
                              \nabla_i\br & \nabla_i\nabla_j\br+2\br P_{ij}
                                \end{pmatrix}.
$$               
For a defining density $\br\in\calE(2)$, we call $\wt g_{IJ}=D_ID_J\br$ a {\it pre-ambient metric}. By 
strict convexity of $M$, $\wt g_{IJ}$ is identified with  a Lorentz metric on $\wt N$ near $M$ which is  homogeneous of degree 2. Then the Monge--Amp\`{e}re equation is equivalent to $\det\wt g_{IJ}=-1$. Note, however, that the equation does not imply that $\wt g_{IJ}$ is Ricci flat, unlike in the complex case. 

The real Monge--Amp\`{e}re equation can be rewritten as an equation for a density $\bu\in\calE(1)|_\Omega$ which is continuous up to $M$ but may not be a defining density.
Setting $\bu=-(-2\br)^{1/2}$, we have
$$
\det(D_ID_J\br)=-(-\bu)^{n+3}\det(\nabla_i\nabla_j\bu+P_{ij}\bu).
$$
Thus \eqref{MA1} is equivalent to 
\begin{eqnarray}\label{MA2}
\begin{cases}
(-\bu)^{n+3}\det(\nabla_i\nabla_j\bu+P_{ij}\bu)=1, \\
\bu|_{M}=0.
\end{cases}
\end{eqnarray}
The differential operator $\nabla_i\nabla_j+P_{ij}$ on $\calE(1)$ is the projecting part of $D_ID_J$ so it is projectively invariant. It is also known as the first BGG operator associated with the representation of 
$SL(n+2, \R)$ on $(\R^{n+2})^\ast$.
The generalization of real Monge--Amp\`{e}re equation via the first BGG operator is also considered by Fox in \cite{Fo}.

With a solution $\br$ or $\bu$ to the Monge--Amp\`{e}re equation, we define a projectively invariant metric on $\Omega$ by
\begin{equation}\label{Blaschke}
\begin{aligned}
g_{ij}&=\frac{\nabla_i\nabla_j\bu}{-\bu}-P_{ij} \\
&=\frac{\nabla_i\nabla_j\br}{-2\br}+\frac{\nabla_i\br\nabla_j\br}{4\br^2}-P_{ij}\in\calE_{(ij)}. 
\end{aligned}
\end{equation}
We call $g_{ij}$ the {\it Blaschke metric} by following \cite{Sa1}. In the case where $N=\R^{n+1}$ and 
$[\nabla]$ is the canonical flat projective structure, unique existence of the solution to \eqref{MA2}
 has been proved by Cheng--Yau:
\begin{thm}[\cite{CY1}]
Let $\Omega$ be a bounded strictly convex domain in $\R^{n+1}$. Then there exists a unique 
function $u\in C^{\infty}(\Omega)\,\cap\, C(\overline{\Omega})$ such that $(-u)^{n+3}\det(\pa^2u/\pa x^i\pa x^j)\\
=1$ and $u|_{\partial\Omega}=0$.
\end{thm}
For $\Omega={\mathbb{B}}^{n+1}$, the solution is given by $u(x)=-(1-|x|^2)^{1/2}$ and the Blaschke metric agrees with the Klein model hyperbolic metric. 
\subsection{Fefferman's approximate solution}\label{Feff-approx}
We do not have the existence theorem for the Monge--Amp\`{e}re equation on convex domains in general projective manifolds, but we can construct an asymptotic solution to \eqref{MA1} by the method of Fefferman \cite{Fe} (see also \cite{Sa2} for the real case).
\begin{lem}\label{MAlem}
Let $\br\in\calE(2)$ be a defining density of $\Omega$. \\
{\rm (i)} For any positive function $f\in\calE$, it holds that
\begin{equation}\label{var-MA1}
\calJ[f\br]=f^{n+2}\calJ[\br]+O(\br).
\end{equation}
{\rm (ii)} Let $m\ge 2$ and suppose that $\br$ satisfies $\calJ[\br]=-1+K\br^{m-1}$ with $K\in\calE(2-2m)$.
Then for any density $\phi\in\calE(2-2m)$, it holds that
\begin{equation}\label{var-MA2}
\calJ[\br+\phi\br^m]=-1+\bigl(K-m(n+4-2m)\phi\bigr)\br^{m-1}+O(\br^m).
\end{equation}
Moreover, if $m\ge3$ we also have
\begin{equation}\label{var-MA3}
\calJ[\br+\phi\br^m]=-1+\bigl(K-m(n+4-2m)\phi\bigr)\br^{m-1}+\br^m\wt\Delta\phi+O(\br^{m+1}),
\end{equation}
where $\wt\Delta=-\wt g^{IJ}\wt\nabla_I\wt\nabla_J$ is the Laplacian of $\wt g_{IJ}=D_ID_J\br$.
\end{lem}
\begin{proof}
(i) Using row and column transformations, we have
\begin{align*}
\calJ[f\br]&=\det
\begin{pmatrix}
0 & f\nabla_j\br \\
f\nabla_i\br & f\nabla_i\nabla_j\br+2\nabla_{(i}f\nabla_{j)}\br
\end{pmatrix}+O(\br) \\
&=\det\begin{pmatrix}
0 & f\nabla_j\br \\
f\nabla_i\br & f\nabla_i\nabla_j\br
\end{pmatrix}+O(\br) \\
&=\det\begin{pmatrix}
2f\br & f\nabla_j\br \\
f\nabla_i\br & f(\nabla_i\nabla_j\br+2\br P_{ij})
\end{pmatrix}+O(\br) \\
&= f^{n+2}\calJ[\br]+O(\br).
\end{align*}
Thus we obtain \eqref{var-MA1}. \\
(ii) We note that $\wt\Delta\phi=-\wt g^{IJ}D_ID_J\phi+O(\br)$ holds, since the Christoffel symbols $\wt\Gamma_{IJ}{}^K$ satisfy $\wt g^{IJ}\wt\Gamma_{IJ}{}^K=O(\br)$ by the Monge--Amp\`ere equation. 
In the following calculation, we compute modulo $O(\br^{m+1})$ for $m\ge 3$ and modulo $O(\br^2)$ for $m=2$. First, we divide each entry in the matrix $D_ID_J(\br+\phi\br^m)$ by $1+m\phi\br^{m-1}$ and perform some row and column transformations to obtain 
\begin{align*}
\calJ[\br+\phi\br^m]
&\equiv(1+m\phi\br^{m-1})^{n+2} \\
&\quad\cdot
\det\begin{pmatrix}
2\br+2(1-m)\phi\br^m & \bigl(1-2m(m-1)\phi\br^{m-1}\bigr)\nabla_j\br \\
\quad & +(1-2m)\br^m\nabla_j\phi \\
\quad & \quad \\
\nabla_i\br+(1-2m)\br^m\nabla_i\phi & \nabla_i\nabla_j\br+2\br P_{ij} \\
\quad&  +\br^m\bigl(\nabla_i\nabla_j\phi+2(1-m)\phi P_{ij}\bigr)
\end{pmatrix}.
\end{align*}
Then, dividing the first row by $1-2m(m-1)\phi\br^{m-1}$, we have
\begin{align*}
\calJ[\br+\phi\br^m]&\equiv\bigl(1+m(n+2)\phi\br^{m-1}\bigr)\bigl(1-2m(m-1)\phi\br^{m-1}\bigr) \\
&\quad\cdot\det
\begin{pmatrix}
2\br+2(1-m)(1-2m)\phi\br^m & \nabla_j\br +(1-2m)\br^m\nabla_j\phi  \\
\quad & \quad \\
\nabla_i\br+(1-2m)\br^m\nabla_i\phi & \nabla_i\nabla_j\br+2\br P_{ij} \\
\quad&  +\br^m\bigl(\nabla_i\nabla_j\phi+2(1-m)\phi P_{ij}\bigr)
\end{pmatrix} \\
&\equiv\bigl(1+m(n+4-2m)\phi\br^{m-1}\bigr)\det(\wt g_{IJ}+\br^m D_ID_J\phi) \\
&\equiv\bigl(1+m(n+4-2m)\phi\br^{m-1}\bigr)\calJ[\br](1-\br^m\wt\Delta\phi) \\
&\equiv -1+\bigl(K-m(n+4-2m)\phi\bigr)\br^{m-1}+\br^m\wt\Delta\phi.
\end{align*}
Thus we obtain \eqref{var-MA2} and \eqref{var-MA3}.
\end{proof}

\begin{prop}
There exists a defining density $\br\in\calE(2)$ such that
\begin{equation}\label{approxMA}
\calJ[\br]=\begin{cases}
-1+O(\br^{\infty}) & {\text if }\ n\ {\text is\ odd}, \\
-1+\calO\br^{n/2+1} &  {\text if }\ n\ {\text is\ even},
\end{cases}
\end{equation}
where $\calO\in\calE(-n-2)$. Moreover, such a density is unique modulo $O(\br^{\infty})$ for $n$ odd, and unique modulo $O(\br^{n/2+2})$ for $n$ even.
\end{prop}
\begin{proof}
Take an arbitrary defining density $\br\in\calE(2)$. Since $M$ is strictly convex, we have $\calJ[\br]<0$ near $M$. Using \eqref{var-MA1} with $f=(-\calJ[\br])^{-1/(n+2)}$, we obtain a defining density $\br^\prime$ such that $\calJ[\br^\prime]=-1+O(\br^\prime)$. Then we inductively improve the approximate solution by using \eqref{var-MA2} up to $m=n/2+1$ when $n$ is even and for all $m$ when $n$ is odd. We use Borel's lemma to obtain the infinite order approximate solution for $n$ odd. The uniqueness up to the stated order also follows from \eqref{var-MA2}.
\end{proof}

We call a defining density which satisfies \eqref{approxMA} a {\it Fefferman defining density}. By Lemma \ref{MAlem}, we can show that $\calO|_M$ is independent of the choice of a Fefferman defining density and it vanishes if and only if there exists $\br$ such that $\calJ[\br]=-1+O(\br^{\infty})$. We call $\calO$ the {\it obstruction density}.

A pre-ambient metric $\wt g_{IJ}=D_ID_J\br$ is called the {\it ambient metric} if $\br$ is a Fefferman defining density.

\subsection{Volume expansion of the Blaschke metric}
Let  $\br\in\calE(2)$ be a Fefferman defining density and set $\bu=-(-2\br)^{1/2}\in\calE(1)$. We define the Blaschke metric $g_{ij}$ by \eqref{Blaschke} near the boundary and extend it to a Riemannian metric on 
$\Omega$. Then the volume expansion of $g_{ij}$ is given by the following theorem:

\begin{thm}\label{volume}
Let $\tau\in\calE(1)$ be a projective scale and put $u:=\tau^{-1}\bu\in\calE$. Then we have the expansion
\begin{equation}\label{volexp}
{\rm Vol}(\{u<-\e\}) 
=\sum_{j=0}^{\lceil n/2\rceil-1}c_{-n+2j}\e^{-n+2j}+
\begin{cases}V+o(1) & (n: {\rm odd}) \\
L\log\frac{1}{\e}+V+o(1) & (n: {\rm even}).
\end{cases}
\end{equation}
The constant term $V$ is independent of the choice of $\tau$ when $n$ is odd while $L$ is independent of 
the choice of $\tau$ when $n$ is even.
\end{thm}
We call the constant term $V$ the {\it renormalized volume} of $g$.
\begin{proof}[Proof of Theorem \ref{volume}]
The approximate Monge--Amp\`{e}re equation \eqref{approxMA} is equivalent to 
$$
(-u)^{n+3}\det{}_\tau(\nabla_i\nabla_j u+P_{ij}u)=1+O(u^m),
$$
where $m=n+2$ for $n$ even and $m=\infty$ for $n$ odd. Thus we have 
$$
\det{}_\tau g_{ij}=(-u)^{-2n-4}(1+O(u^m)).
$$
Near the boundary we identify $\Omega$ with the product $M\times (-\e_0, 0)$, where the second component is given by $\r=\tau^{-2}\br$. 
Take a volume form $vol_M$ on $M$, and write $\tau^{-(n+2)}=v d\r\wedge vol_M$ with 
$v\in C^{\infty}(M\times (-\e_0, 0])$. Then since $v$ is even in $u$, we have
\begin{align*}
vol_g&=(\det{}_\tau g_{ij})^{1/2}\tau^{-(n+2)} \\
&=(-u)^{-n-2}(1+O(u^m))v d\r\wedge vol_M \\
&=\Bigl(\,\sum_{j=0}^{[n/2]}a_{j}u^{-n-1+2j}+O(1)\Bigr)du\wedge vol_M.
\end{align*}
Integrating both sides, we obtain the desired volume expansion. 

Let $\wh{\tau}=e^{-\U}\tau$ be another projective scale, and set $\wh\r={\wh\tau}^{-2}\br, \wh u={\wh\tau}^{-1}\bu$. Near the boundary, $u$ and $\wh u$ are related as $u=b(x, \wh\r\,)\wh u$ with some
$b\in C^{\infty}(M\times(-\wh\e_0, 0])$. Then the difference of the volumes of sublevel sets is 
\begin{multline*}
{\rm Vol}(\{\wh u<-\e\})-{\rm Vol}(\{u<-\e\}) \\
=\int_{-\e}^{-b(x, -\frac{1}{2}\e^2)\e}\int_M \Bigl(\,\sum_{j=0}^{[n/2]}a_{j}u^{-n-1+2j}+O(1)\Bigr)du\wedge vol_M.
\end{multline*}
When $n$ is odd, we can see that this integral is odd in $\e$ so it does not contain the constant term. Therefore, $V$ is invariant under rescaling of $\tau$.
When $n$ is even, the difference does not contain the logarithmic term, so $L$ is invariant.
\end{proof}

\section{Relation to geometry on the boundary}
\subsection{Geometry of hypersurfaces in projective manifolds}
Let $(N, [\nabla])$ be an oriented projective manifold of dimension $n+1$, and let $\Omega\subset N$ be a relatively compact domain with strictly convex boundary $M$. We assume that the projective structure $[\nabla]$ is locally flat. We review the geometric structure induced on the hypersurface $M$; see, e.g., \cite{NS} for details.

We fix a projective scale $\tau\in\calE(1)$ and a defining function $\r$. Take a vector field $\xi\in\Gamma(M, TN)$ with $d\r(\xi)>0$. Then according to the decomposition $TN|_M=TM\oplus\R\xi$, we have
\begin{align*}
\nabla_X Y&=\nabla^\xi_X Y-h(X, Y)\xi, \\
\nabla_X \xi&=S(X)+\eta(X)\xi          
\end{align*}
for $X, Y\in\Gamma(TM)$. Since $\nabla$ is torsion-free, $\nabla^\xi$ defines a torsion-free affine connection on $TM$ and $h$ becomes a symmetric 2-tensor, called the {\it affine second fundamental form}. Moreover, strict convexity of $M$ implies that $h$ is positive definite, so it is  called the {\it affine metric}. We denote the Levi--Civita connection of $h$ and its curvature by $\nabla^h$ and $R^h$ respectively. The endomorphism $S$ is called the {\it affine shape operator}. These quantities depend on the choice of the transverse field $\xi$. However we can take $\xi$ canonically by imposing some normalization conditions:
\begin{lem}\label{affine-normal}
There exists a unique $\xi\in\Gamma(M, TN)$ with $d\r(\xi)>0$ such that
$\bigl(\xi\lrcorner\,\tau^{-(n+2)}\bigr)|_{TM}=vol_h$ and $\eta=0$.
\end{lem}
\begin{proof}
We take an arbitrary $\xi\in\Gamma(M, TN)$ with $d\r(\xi)>0$. Then any other $\xi^\prime\in\Gamma(M, TN)$ such that $d\r(\xi^\prime)>0$ can be written in the form $\xi^\prime=\phi\xi+Z$ with $\phi>0$ and $Z\in\Gamma(TM)$. A calculation shows that the corresponding tensor fields are given by $h^\prime=\phi^{-1}h$ and $\eta^\prime(X)=\eta(X)+\phi^{-1}X\phi-\phi^{-1}h(X, Z)$. Thus we have 
$vol_{h^{\prime}}=\phi^{-n/2}vol_h$. On the other hand we have $\bigl(\xi^{\prime}\lrcorner\,\tau^{-(n+2)}\bigr)|_{TM}=\phi\bigl(\xi\lrcorner\,\tau^{-(n+2)}\bigr)|_{TM}$, so there is a unique $\phi$ such that $\bigl(\xi^\prime\lrcorner\,\tau^{-(n+2)}\bigr)|_{TM}=vol_{h^\prime}$. Then since $h$ is non-degenerate, the condition $\eta^\prime=0$ uniquely determines $Z$.
\end{proof}
The transverse vector field $\xi$ in Lemma \ref{affine-normal} is called the {\it affine normal field} associated with the projective scale $\tau$. The transformation laws under rescaling of $\tau$ are given by the following
\begin{prop}\label{affine-transform}
Let $\wh\tau=e^{-\U}\tau$ be another projective scale. Then it holds that 
\begin{align*}
\wh\xi&=e^{-2\U}(\xi-{\rm grad}_h\U), \\
\wh h&=e^{2\U}h, \\
\wh\nabla^{\wh\xi}_X Y&=\nabla^\xi_ X Y+d\U(X)Y+d\U(Y)X-h(X, Y)\,{\rm grad}_h\U, \\
\wh S(X)&=e^{-2\U}\bigl(S(X)+(\xi\U-|d\U|_h^2)X+d\U(X)\,{\rm grad}_h\U-\nabla^\xi_X {\rm grad}_h\U\bigr),
\end{align*}
where ${\rm grad}_h\U$ denotes the gradient vector field of $\U|_M$ with respect to the metric $h$.
\end{prop}
Since the affine metric transforms conformally under a change of $\tau$, the conformal structure $[h]$ is well-defined. Recall that the conformal density bundle on $M$ is defined by $\calE[1]=(\wedge^{n} TM)^{-1/n}$. Using affine metric we can define a canonical isomorphism $\calE(1)|_M\rightarrow\calE[1]$ by 
$\tau|_M\mapsto (vol_h)^{-1/n}$. Thus we identify $\calE(w)|_M$ with $\calE[w]$ for each $w$. The conformally invariant weighted tensor $\bh_{\a\b}=\tau^2|_M h_{\a\b}\in\calE_{(\a\b)}[2]$ is called the {\it conformal metric}. We raise and lower the indices with $\bh_{\a\b}$ and its inverse $\bh^{\a\b}$. We put conformal weight $-2$ on the affine shape operator and write $\bS_\a{}^\b=\tau|_M^{-2}S_\a{}^\b\in\calE_\a{}^\b[-2]$. Note that $\bS_\a{}^\b$ depends on 1-jet of the scale along the boundary.

From the transformation laws of the induced connection $\nabla^\xi$ and the Levi--Civita connection $\nabla^h$, it follows that their difference $A_{\a\b}{}^\g=\nabla^h_\a-\nabla^\xi_\a$, called the {\it Fubini--Pick form}, is a conformally invariant tensor. This tensor is a fundamental invariant of $M$, and it is a classical theorem that the Fubini--Pick form vanishes identically if and only if $M$ is locally equivalent to the sphere (see \cite{NS} for a proof):
\begin{thm}[Blaschke--Pick--Berwald]
Let $M$ be a strictly convex hypersurface in a locally flat projective manifold $N$ of dimension $n+1$. Then the Fubini-Pick form vanishes identically if and only if, for each point $p\in M$, there exist a neighborhood $U\subset N$, an open subset $V\subset\R^{n+1}$, and a projective equivalent map $\varphi:U\rightarrow V$ such that $\varphi(U\cap M)=V\cap S^n$. 
\end{thm}
As a consequence of the flatness of the projective structure $[\nabla]$, $A_{\a\b\g}$ and $\bS_{\a\b}$ have the following symmetries:
\begin{prop}\label{sym-A-S}
The Fubini--Pick form $A_{\a\b\g}$ and the affine shape operator $\bS_{\a\b}$ satisfy
$$
A_{\a\b\g}=A_{(\a\b\g)}, \qquad A_{\a\g}{}^\g=0, \qquad \bS_{[\a\b]}=0.
$$
\end{prop}
\begin{proof}
Take a local frame $\{e_\a\}$ for $TM$. Then $\{e_\infty=\xi, e_\a\}$ forms a local frame for $TN|_M$. Let $\{\th^\infty, \th^\a\}$ be the dual frame. Since $\xi$ is affine normal, the connection forms of $\nabla$ satisfy
$$
\omega_\a{}^\b|_{TM}=\omega^\xi{}_\a{}^\b, \quad  \omega_\infty{}^\b|_{TM}=S_\g{}^\b\th^\g, \quad
\omega_\a{}^\infty|_{TM}=-h_{\a\g}\th^\g,  \quad \omega_\infty{}^\infty|_{TM}=0,
$$
where $\omega^\xi{}_\a{}^\b$ is the connection forms of $\nabla^\xi$. Thus the curvature forms satisfy
$$
\Omega_\a{}^\infty|_{TM}=-\nabla^\xi_\mu h_{\g\a}\th^\mu\wedge\th^\g, \qquad
\Omega_\infty{}^\infty|_{TM}=-h_{\mu\g}S_{\a}{}^\mu\th^\g\wedge\th^\a.
$$
On the other hand, from $C_{ij}{}^k{}_l=0$ we have $R_{ij}{}^k{}_l=2\d_{[i}{}^k P_{j]l}$ so that 
$\Omega_\a{}^\infty|_{TM}=\Omega_\infty{}^\infty|_{TM}=0$. Therefore we obtain $\nabla^\xi_{[\mu} h_{\g]\a}=0$ and $\bS_{[\a\g]}=0$. It follows that $A_{\a\b}{}^\g=(1/2)h^{\g\mu}\nabla^\xi_\a h_{\b\mu}$ and $A_{\a\b\g}\in\calE_{(\a\b\g)}[2]$. To prove that $A_{\a\b\g}$ is trace-free, it suffices to show $\nabla^\xi vol_h=0$.
For any $X, X_1, \dots, X_n\in\Gamma(TM)$, we have
\begin{align*}
0&=(\nabla_X\tau^{-(n+2)})(\xi, X_1, \dots, X_n) \\
&=X\cdot\tau^{-(n+2)}(\xi, X_1,\dots, X_n)-\tau^{-(n+2)}(\nabla_X \xi, X_1,\dots,X_n) \\
&\quad-\sum_{i=1}^{n}\tau^{-(n+2)}(\xi, X_1,\dots, \nabla_X X_i, \dots, X_n) \\
&=X\cdot vol_h(X_1,\dots,X_n)-\sum_{i=1}^{n}vol_h (X_1,\dots, \nabla^\xi_X X_i, \dots, X_n) \\
&=(\nabla^\xi _X vol_h)(X_1,\dots,X_n).
\end{align*}
Thus $A_{\a\b\g}$ is trace-free.
\end{proof}

\subsection{Volume expansion and affine invariants}\label{vol-exp-affine}
By an {\it affine invariant} of the boundary, we mean a linear combination of complete contractions of $h_{\a\b}$, $R^h_{\a\b\g\mu}$, $A_{\a\b\g}$, $S_{\a\b}$ and their covariant derivatives with respect to $\nabla^h$ in a chosen projective scale. Recall that affine invariants depend only on 1-jet of the projective scale along the boundary. In this subsection we prove that, given a 1-jet of the projective scale along the boundary, there exists an appropriate extension of the scale such that the coefficients $c_j$ and $L$ in \eqref{volexp} can be written as the integrals of some affine invariants over $M$. 

Let $\br\in\calE(2)$ be a Fefferman defining density and $g_{ij}$ be the associated Blaschke metric on $\Omega$. Until the normalization of the projective scale is involved, we work in an arbitrary projective scale $\tau\in\calE(1)$. Set $\r=\tau^{-2}\br$ and let $\nabla\in[\nabla]$ be the representative connection such that $\nabla\tau=0$.
By strict convexity of $M$, there exists a unique $\xi\in \Gamma(TN)$ near $M$ which satisfies 
\begin{equation}\label{xi}
\xi \r=1, \qquad \xi^i X^j\nabla_i\nabla_j\r=0 \quad {\rm if}\quad X\r=0.
\end{equation}
We call $r:=\xi^i\xi^j\nabla_i\nabla_j\r$ the {\it transverse curvature}. If we take a local frame $\{e_\a\}$ for the subbundle ${\rm Ker}\, d\r\subset TN$, the tuple $\{e_\infty=\xi, e_\a\}$ forms a local frame for $TN$. Let 
$\{ \th^\infty=d\r, \th^\a\}$ be the dual frame. Then we have
\begin{equation}\label{nabla-dr}
\nabla d\r=h_{\a\b}\th^\a\otimes\th^\b+rd\r\otimes d\r,
\end{equation}
where $h_{\a\b}$ is the affine metric on each level set of $\r$. The ambient metric $\wt g_{IJ}$ satisfies 
\begin{equation}\label{g-M}
\wt g_{IJ}|_M
=\begin{pmatrix}
0 & 1 & \quad \\
1 & \boldsymbol{r}|_M & \quad \\
\quad & \quad & \bh_{\a\b}
\end{pmatrix},
\quad 
\wt g^{IJ}|_M
=\begin{pmatrix}
-\boldsymbol{r}|_M  & 1 & \quad \\
1 & 0 & \quad \\
\quad & \quad & \bh^{\a\b}
\end{pmatrix},
\end{equation}
where $\boldsymbol{r}:=\tau^{-2}r$.
\begin{lem}\label{xi-affine-normal}
The vector filed $\xi|_M$ agrees with the affine normal field.
\end{lem}
\begin{proof}
First, using the characterization \eqref{xi} we have 
\begin{align*}
0&=X^i\xi^j\nabla_i\nabla_j\r \\
&=X^i\nabla_i(\xi^j\nabla_j \r)-X^i\nabla_i\xi^j \nabla_j\r \\
&=-d\r(\nabla_X\xi)
\end{align*}
for $X\in\Gamma(TM)$. This implies $\eta=0$. Next, if we take $\{\xi, e_\a\}$ so that the dual frame $\{d\r, \th^\a\}$ satisfies  
 $\tau^{-(n+2)}=d\r\wedge\th^1\wedge\dots\wedge\th^n$, we have
$$
-1=\calJ[\br]|_M =\det 
\begin{pmatrix}
0 & 1 & \quad \\
1 & r & \quad \\
\quad & \quad & h_{\a\b}
\end{pmatrix} 
=-\det h_{\a\b}.
$$
Thus, at the boundary it holds that 
$$
vol_h=(\det h_{\a\b})^{1/2}\th^1\wedge\dots\wedge\th^n=\th^1\wedge\dots\wedge\th^n=\xi\lrcorner\,\tau^{-(n+2)}.
$$
Therefore $\xi|_M$ is the affine normal field.
\end{proof}

Let $F: M\times (-\e_0, 0]\longrightarrow U\subset\overline{\Omega}$ be the diffeomorphism defined by the flow generated by $\xi$. Since $\xi\r=1$ and $\r|_M=0$, it holds that the second component of $F^{-1}$ equals $\r$ and $F_\ast\pa/\pa\r=\xi$. We regard differential forms on $M$ as those on $M\times(-\e_0, 0]$ by pulling them back with the first projection. Take a local frame $\{e_\a\}$ for ${\rm Ker}\,d\r$ such that the frame $\{e_\infty=\xi, e_\a\}$ is unimodular with respect to $\tau^{-(n+2)}$, and let $\{\th^\infty=d\r, \th^\a\}$ be the dual frame. Then we can write as 
\begin{align*}
F^\ast\tau^{-(n+2)} &=F^\ast(d\r\wedge\th^1\wedge\dots\wedge\th^n) \\
&=v(x, \r)\,d\r\wedge\th^1|_{TM}\wedge\dots\wedge\th^n|_{TM} \\
&=v(x, \r)\,d\r\wedge vol_h,
\end{align*}
for a function $v$ with $v(x, 0)=1$. In view of the proof of Theorem \ref{volume}, the coefficients $c_j$ and $L$ are the integrals of the Taylor coefficients in $\r$ of $v(x, \r)$ at the boundary. In order to obtain the expansion of $v$, we will find a differential equation which $v$ satisfies. Taking exterior derivatives of both sides of 
$$
F^\ast (\th^1\wedge\dots\wedge\th^n)=v(x, \r)\,vol_h,
$$
we have
\begin{equation}\label{pull-back-tau}
F^\ast d(\th^1\wedge\dots\wedge\th^n)=\frac{\pa v}{\pa\r}\,d\r\wedge vol_h.
\end{equation}
We write the connection forms of $\nabla$ as $\omega_j{}^i=\Gamma_{kj}{}^i\th^k$. Since $\nabla$ is torsion-free, it holds that $d\th^i=\th^j\wedge\omega_j{}^i$, and thus we compute as
\begin{align*}
d(\th^1\wedge\dots\wedge\th^n)&=-\omega_\g{}^\g\wedge\th^1\wedge\dots\wedge\th^n
+\Gamma_{\g\infty}{}^\g d\r\wedge\th^1\wedge\dots\wedge\th^n \\
&=(\Gamma_{\g\infty}{}^\g-\Gamma_{\infty\g}{}^\g)\,d\r\wedge\th^1\wedge\dots\wedge\th^n.
\end{align*}
From $\nabla(d\r\wedge\th^1\wedge\dots\wedge\th^n)=0$ it follows that $\Gamma_{\infty k}{}^k=0$, so
we have $\Gamma_{\infty\g}{}^\g=-\Gamma_{\infty\infty}{}^\infty=r$ by \eqref{nabla-dr}. Therefore we obtain 
\begin{equation}\label{pull-back-tau2}
F^\ast d(\th^1\wedge\dots\wedge\th^n)=F^\ast({\rm tr}S-r)v\, d\r\wedge vol_h,
\end{equation}
where $S$ is the affine shape operator on each level set of $\r$. From \eqref{pull-back-tau} and 
\eqref{pull-back-tau2}, we obtain the following differential equation for $v$:
\begin{equation}\label{eq-v}
\begin{cases}
\xi v=({\rm tr}S-r)v, \\
v|_M=1.
\end{cases}
\end{equation}
Here, we have regarded $v$ as a function on $U$. Thus the expansion of $v$ is reduced to those of $r$ and ${\rm tr}S$. 
We will work in a special local frame:
\begin{lem}
Let $\{e_\a\}$ be a local frame for $TM$ which is unimodular with respect to $h$. If one extends $\{e_\a\}$ by 
the parallel transport relative to $\nabla$ along the integral curves of $\xi$, then it becomes a local frame for 
${\rm Ker}\, d\r\subset TN$.
\end{lem}
\begin{proof}
We set $f=d\r(e_\a)$ for fixed $\a$. By $\nabla_\xi e_\a=0$ and \eqref{nabla-dr}, we have 
$\xi f =(\nabla_\xi d\r)(e_\a)=rf$. Thus $f$ is the solution to the ODE $\xi f=rf$, $f|_{M}=0$. It follows that $f=0$ and so $e_\a\in{\rm Ker}\,d\r$.  
\end{proof}
We call $\{e_\infty=\xi, e_\a\}$ an {\it adapted frame} for $TN$ when $\{e_\a\}$ is given as above. In an adapted frame, the connection forms of $\nabla$ are given by 
\begin{equation}\label{conn-adapted}
\begin{aligned}
\omega_\a{}^\b&=\omega^\xi{}_\a{}^\b, & \omega_\infty{}^\a&=S_\g{}^\a\th^\g+C^\a d\r, \\ 
\omega_\b{}^\infty&=-h_{\b\g}\th^\g, & \omega_\infty{}^\infty&=-r d\r,
\end{aligned}
\end{equation}
where $C^\a$ is a function, and $\omega^\xi{}_\a{}^\b$ restricts to the connection form of the induced connection on each level set and satisfies $\omega^\xi{}_\a{}^\b(\xi)=0$. 
In the sequel, we will derive equations for the computation of the expansions of $S_\a{}^\b$, $h_{\a\b}$, $C^\a$, $r$, and $P_{ij}$ in a fixed adapted frame.

First, the structure equation $\Omega_i{}^j=d\omega_i{}^j-\omega_i{}^k\wedge\omega_k{}^j$ gives
\begin{align*}
\Omega_\infty{}^\b&=\bigl(\xi \,S_\a{}^\b+S_\a{}^\g S_\g{}^\b+r S_\a{}^\b-\nabla^\xi_\a C^\b\bigr)d\r\wedge\th^\a+\nabla^\xi_\a S_\g{}^\b\th^\a\wedge\th^\g, \\
\Omega_\a{}^\infty&=\bigl(\xi\, h_{\a\b}-rh_{\a\b}+h_{\a\g}S_\b{}^\g\bigr)\th^\b\wedge d\r
-\nabla^\xi_\b h_{\a\g}\th^\b\wedge\th^\g, \\
\Omega_\infty{}^\infty&=\bigl(C^\g h_{\g\a}+r_\a\bigr)d\r\wedge\th^\a+S_\g{}^\mu h_{\mu\b}\th^\g\wedge\th^\b, \\
\Omega_\a{}^\b&\equiv\Bigl(\frac{1}{2}R^\xi_{\g\mu}{}^\b{}_\a+h_{\a\g}S_\mu{}^\b\Bigr)\th^\g\wedge\th^\mu\quad {\rm mod}\ d\r,
\end{align*}
where $R^\xi_{\g\mu}{}^\b{}_\a$ is the curvature tensor of $\nabla^\xi$. Comparing the coefficients of both sides of the above equations and using $C_{ij}{}^k{}_l=R_{ij}{}^k{}_l-2\d_{[i}{}^kP_{j]l}=0$, we have
\begin{align}
\xi\, S_\a{}^\b+S_\a{}^\g S_\g{}^\b+r S_\a{}^\b-\nabla^\xi_\a C^\b+\d_\a{}^\b P_{\infty\infty}&=0, 
\label{xi-S} \\
\nabla^\xi_{[\a}S_{\g]}{}^\b-\d_{[\a}{}^\b P_{\g]\infty}&=0, \label{nabla-S} \\
\xi\, h_{\a\b}-rh_{\a\b}+h_{\a\g}S_\b{}^\g+P_{\a\b}&=0,  \label{xi-h} \\
C^\g h_{\g\a}+r_\a-P_{\a\infty}&=0, \label{C} \\
R^\xi_{\g\mu}{}^\b{}_\a+2h_{\a[\g}S_{\mu]}{}^\b-2\d_{[\g}{}^\b P_{\mu]\a}&=0. \label{Pab}
\end{align}
The contractions of \eqref{nabla-S} and \eqref{Pab} yield
\begin{align}
(n-1)P_{\a\infty}+\nabla^\xi_\a({\rm tr}S)-\nabla^\xi_\g S_{\a}{}^\g&=0, \label{P-inf} \\
(n-1)P_{\a\b}-{\rm Ric}^\xi_{\a\b}-S_{\a\b}+({\rm tr}S)h_{\a\b}&=0, \label{Pab-Ric}
\end{align}
where ${\rm Ric}^\xi_{\a\b}$ is the Ricci tensor of $\nabla^\xi$.

Next, we take the exterior derivatives of both sides of 
$$
\Omega_\a{}^\b=d\omega_\a{}^\b-\omega_\a{}^\g\wedge\omega_\g{}^\b-\omega_\a{}^\infty\wedge\omega_\infty{}^\b
$$
and compare the coefficients of $d\r\wedge\th^\g\wedge\th^\mu$ to obtain 
$$
\d_{[\g}{}^\b \xi P_{\mu]\a}+\d_{[\g}{}^\b S_{\mu]}{}^\nu P_{\nu\a}-\d_{[\g}{}^\b\nabla^\xi_{\mu]} P_{\a\infty}
+h_{\a[\g}\d_{\mu]}{}^\b P_{\infty\infty}=0.
$$
Taking the trace of the equation, we have
\begin{equation}\label{xi-P}
\xi P_{\a\b}+S_\a{}^\g P_{\g\b}-\nabla^\xi_\a P_{\b\infty}-P_{\infty\infty}h_{\a\b}=0.
\end{equation}

We  shall derive an equation for $r$ from the Monge--Amp\`{e}re equation. If we write 
$$
\tau^{-(n+2)}=e^\varphi d\r\wedge\th^1\wedge\cdots\wedge\th^n
$$
with a function $\varphi$, we have
$$
0=\nabla_\xi \tau^{-(n+2)}=(\xi \varphi-\Gamma_{\infty j}{}^j)\tau^{-(n+2)}=(\xi \varphi+r)\tau^{-(n+2)}
$$
and thus
\begin{equation}\label{xi-phi}
\xi \varphi=-r.
\end{equation}
Let $\psi_i{}^j$ be the Levi--Civita connection forms of $g_{ij}$ with respect to the frame $\{e_\infty, e_\a\}$. Then, since
$$
vol_g=\det{}_\tau(g_{\ij})^{1/2}\,\tau^{-(n+2)}=\det{}_\tau(g_{ij})^{1/2}\,e^{\varphi}d\r\wedge\th^1\wedge\cdots\wedge\th^n,
$$
we have 
$$
\psi_k{}^k=d\log\bigl(\det{}_\tau(g_{ij})^{1/2}\,e^{\varphi}\bigr)=\frac{1}{2}d\log\det{}_\tau(g_{ij})+d\varphi,
$$
which gives
\begin{equation}\label{d-vol}
\xi\log\det{}_\tau(g_{ij})=2\psi_k{}^k(\xi)+2r. 
\end{equation}
by \eqref{xi-phi}. In the adapted frame, the Blaschke metric is given by 
$$
g=\frac{\wt h_{\a\b}}{-2\r}\,\th^a\cdot\th^\b-2P_{\a\infty}\,\th^\a\cdot d\r+\frac{1-2r\r-4\r^2P_{\infty\infty}}{4\r^2}\,d\r^2,
$$
where $\wt h_{\a\b}=h_{\a\b}+2\r P_{\a\b}$. Therefore, setting $\wt h^{\a\b}=(\wt h_{\a\b})^{-1}$, we have 
\begin{align*}
2\psi_k{}^k(\xi)&=\xi\log\det
\begin{pmatrix}
(1-2r\r-4\r^2P_{\infty\infty})/4\r^2 & -P_{\b\infty} \\
-P_{\a\infty} & \wt h_{\a\b}/(-2\r)
\end{pmatrix} \\
&=\xi\log\det
\begin{pmatrix}
(1-2r\r-4\r^2P_{\infty\infty})/4\r^2+2\r\wt h^{\g\mu}P_{\g\infty}P_{\mu\infty} & -P_{\b\infty} \\
0 &  \wt h_{\a\b}/(-2\r)
\end{pmatrix} \\
&=-(n+2)\r^{-1}+\xi\log\det(\wt h_{\a\b}) \\
&\quad +\xi\log
\bigl(1-2r\r-4\r^2P_{\infty\infty}+8\r^3\,\wt h^{\g\mu}P_{\g\infty}P_{\mu\infty}\bigr).
\end{align*}
By \eqref{xi-h} we also have 
\begin{align*}
\xi\log\det(\wt h_{\a\b})&=h^{\a\b}\xi\, h_{\a\b}+\xi\log\det\bigl(\d_\a{}^\b+2\r h^{\b\g}P_{\g\a}\bigr) \\
&=nr-{\rm tr}S-h^{\a\b}P_{\a\b}+\xi\log\det\bigl(\d_\a{}^\b+2\r h^{\b\g}P_{\g\a}\bigr).
\end{align*}
Consequently the right-hand side of \eqref{d-vol} is written as 
\begin{align*}
2\psi_k{}^k(\xi)+2r&=-(n+2)\r^{-1}+\xi\log
\bigl(1-2r\r-4\r^2P_{\infty\infty}+8\r^3\,\wt h^{\g\mu}P_{\g\infty}P_{\mu\infty}\bigr) \\
&\quad +\xi\log\det\bigl(\d_\a{}^\b+2\r h^{\b\g}P_{\g\a}\bigr)+(n+2)r-{\rm tr}S-h^{\a\b}P_{\a\b}.
\end{align*}
On the other hand, by the approximate Monge--Amp\`{e}re equation \eqref{approxMA}, we have 
$$
(-2\r)^{n+2}\det{}_\tau(g_{ij})=1-\underline{\calO}\,\r^{n/2+1},
$$
where $\underline{\calO}=\tau^{n+2}\calO$. (We will write the formulas only in the case where $n$ is even; when $n$ is odd, terms which come from $\underline{\calO}\,\r^{n/2+1}$ are replaced by $O(\r^\infty)$.) It follows that the left-hand side of \eqref{d-vol} satisfies
$$
(n+2)\r^{-1}+\xi\log\det{}_\tau(g_{ij})=-\Bigl(\frac{n}{2}+1\Bigl)\underline{\calO}\,\r^{n/2}+O(\r^{n/2+1}).
$$
As a result we obtain 
\begin{multline}\label{r}
(n+2)r+\xi\log
\bigl(1-2r\r-4\r^2P_{\infty\infty}+8\r^3\,\wt h^{\g\mu}P_{\g\infty}P_{\mu\infty}\bigr) \\
+\xi\log\det\bigl(\d_\a{}^\b+2\r h^{\b\g}P_{\g\a}\bigr)-{\rm tr}S-h^{\a\b}P_{\a\b} \\
=-\Bigl(\frac{n}{2}+1\Bigr)\underline{\calO}\,\r^{n/2}+O(\r^{n/2+1}). 
\end{multline}

Finally we consider the expansion of $P_{\infty\infty}$, which involves a normalization of the projective scale.
We set 
$$
\mathcal{S}=\Bigl\{\tau_I=
\begin{pmatrix}
\tau \\
\mu_i
\end{pmatrix}
\in\Gamma(M, \wt\calE_I)
\ \big{|}\ 
\tau>0, \ X^i(\nabla_i\tau-\mu_i)=0\  {\rm for}\ X^i\in TM \Bigr\}.
$$
Note that the condition $X^i(\nabla_i\tau-\mu_i)=0$ is independent of the choice of $\nabla\in [\nabla]$ since 
it is the projecting part of $X^i\nabla_i\tau_I$. The set $\mathcal{S}$ can be viewed as the space of 1-jets of 
projective scales at the boundary $M$, so a choice of an element of $\mathcal{S}$ determines affine invariants on $M$. For each $\tau_I\in\mathcal{S}$, we construct a projective scale near $M$ as follows: Take the affine normal field $\xi$ along the boundary, which is determined by $\tau_I$.  Then we extend $\tau_I$ by the parallel transport with respect to the tractor connection along the geodesic paths of $[\nabla]$ with initial velocity $\xi$. The top slot of the extension of $\tau_I$ is positive near $M$ so it defines a projective scale whose 1-jet is given by $\tau_I$. We call a projective scale constructed in this way a {\it normalized scale}. A global affine scale $\tau$ is a normalized scale since $D_I\tau$ is a parallel tractor by Proposition \ref{affine-scale}. 

\begin{lem}
Let $\tau\in\calE(1)$ be a normalized scale and $\nabla\in[\nabla]$ the corresponding representative connection. Let $c$ be the geodesic of $\nabla$ whose initial velocity is given by the affine normal field $\xi$.
Then the velocity vector field $V^i$ of $c$ satisfies
$$
P_{ij}V^iV^j=0,
$$
where $P_{ij}$ is the projective Schouten tensor of $\nabla$.
\end{lem}
\begin{proof}
Let $\tau_I=(\tau, \mu_i)$ be the extension of the 1-jet of $\tau$ by the parallel transport along $c$. By the normalization condition on  $\tau$ it holds that
\begin{equation}\label{parallel}
0=V^i\nabla_i
\begin{pmatrix}
\tau \\
\mu_j
\end{pmatrix}\overset{\tau}{=}
\begin{pmatrix}
-V^i\mu_i \\
V^i\nabla_i\mu_j+V^iP_{ij}\tau
\end{pmatrix}.
\end{equation}
Since $V^j\nabla_jV^i=0$, the first slot of \eqref{parallel} implies
$$
0=V^j\nabla_j(V^i\mu_i)=V^jV^i\nabla_j\mu_i.
$$
Then it follows from the second slot of \eqref{parallel} that
$$
0=V^iV^j\nabla_i\mu_j+V^iV^jP_{ij}\tau=V^iV^jP_{ij}\tau.
$$
Thus we have $P_{ij}V^iV^j=0$.
\end{proof}
We write $V=V^\infty \xi +V^\a e_\a$ with an adapted frame $\{\xi, e_\a\}$. Then since $\nabla_V V=0$,  $P_{ij}V^iV^j=0$ and $V|_M=\xi$, we have 
\begin{equation}\label{P-infty-infty}
\begin{cases}
V^\infty \xi \,V^\infty+V^\g e_\g V^\infty-r(V^\infty)^2-h_{\g\mu}V^\g V^\mu=0, \\
V^\infty \xi \,V^\a+V^\g\nabla^\xi_\g V^\a+S_\g{}^\a V^\g V^\infty+C^\a(V^\infty)^2=0, \\
(V^\infty)^2P_{\infty\infty}+2P_{\g\infty}V^\g V^\infty+P_{\g\mu}V^\g V^\mu=0, \\
V^\infty|_M=1, \\
V^\a|_M=0. 
\end{cases}
\end{equation}
These equations enable us to compute the expansion of $P_{\infty\infty}$ together with the expansions of $V^\infty$ and $V^\a$.

Now we have completed the derivation of all equations that we use in the calculation of expansions of affine invariants.  We will need some commutation relations in the computations. In the following lemma, we present the commutation relation for a tensor $t_\b{}^\g$, but the equation can be extended in an obvious way for general type of tensors.     
\begin{lem}\label{comm-rel}
Let $\{\xi, e_\a\}$ be an adapted frame. For a tensor field $t_\b{}^\g\in\Gamma({\rm Ker}\,d\r\otimes({\rm Ker}\,d\r)^\ast)$, one has
\begin{align*}
\xi\,\nabla^\xi_\a t_\b{}^\g&=\nabla^\xi_\a(\xi\, t_\b{}^\g)-S_\a{}^\mu\nabla^\xi_\mu t_\b{}^\g
+\bigl(C^\g h_{\a\mu}-\d_\a{}^\g P_{\mu\infty}\bigr)t_\b{}^\mu \\
&\quad -\bigl(C^\mu h_{\a\b} -\d_\a{}^\mu P_{\b\infty}\bigr)t_\mu{}^\g.
\end{align*}
\end{lem}
\begin{proof}
We extend $t_\b{}^\g$ to $t_j{}^k$ by setting $t_\infty{}^\g=t_\b{}^\infty=t_{\infty}{}^\infty=0$. By \eqref{conn-adapted}, we have
$$
\nabla_\infty\nabla_\a t_\b{}^\g=\xi\,\nabla_\a t_\b{}^\g+C^\g\nabla_\a t_\b{}^\infty 
=\xi\,\nabla^\xi_\a t_\b{}^\g-C^\g h_{\a\mu}t_\b{}^\mu.
$$
Using $R_{ij}{}^k{}_l=2\d_{[i}{}^kP_{j]l}$, we also have
\begin{align*}
\nabla_\infty\nabla_\a t_\b{}^\g&=\nabla_\a\nabla_\infty t_\b{}^\g-R_{\infty\a}{}^\mu{}_\b t_{\mu}{}^\g
+R_{\infty\a}{}^\g{}_\mu t_\b{}^\mu \\
&=\nabla^\xi_\a(\xi\,t_\b{}^\g)-S_\a{}^\mu\nabla^\xi_\mu t_\b{}^\g-h_{\a\b}C^\mu t_\mu{}^\g
+P_{\b\infty}t_\a{}^\g-\d_\a{}^\g P_{\mu\infty}t_\b{}^\mu.
\end{align*}
Comparing the above two expressions for $\nabla_\infty\nabla_\a t_\b{}^\g$, we obtain the desired equation.
\end{proof}
\begin{prop}\label{coefficients}
Let $\tau\in\calE(1)$ be a normalized projective scale and $\br\in\calE(2)$ a Fefferman defining density. Then in an adapted frame $\{\xi, e_\a\}$, the boundary values of 
\begin{align*}
\xi^{k+1}S_\a{}^\b, \quad &\xi^{k+1}h_{\a\b}, \quad \xi^{k+1}P_{\a\b}, \quad \xi^{k+1}V^{\infty}, \quad \xi^{k+1}V^\a, \\
\xi^k r, \quad &\ \xi^k P_{\a\infty}, \quad \ \,\xi^kP_{\infty\infty}, \quad \xi^k C^\a
\end{align*}
are determined by \eqref{xi-S}, \eqref{xi-h}, \eqref{C}, \eqref{P-inf}, \eqref{Pab-Ric}, \eqref{xi-P}, \eqref{r}, and \eqref{P-infty-infty} for $k\le n/2-1$ when $n$ is even and for all $k$ when $n$ is odd. Moreover they are expressed in terms of $h_{\a\b}$, $S_\a{}^\b$, ${\rm Ric}^h_{\a\b}$, $A_{\a\b\g}$ and their covariant derivatives with respect to $\nabla^h$.
\end{prop}
The proof is given by the induction on $k$; we differentiate the equations in the $\xi$-direction repeatedly and use Lemma \ref{comm-rel} to obtain the Taylor coefficients as tensors on the boundary. If we differentiate \eqref{r} $k$ times, the coefficient of $\xi^k r$ becomes $n-2k$, so when $n$ is even, $\xi^k r|_M$ is determined up to $k=n/2-1$. When $k=n/2$, we obtain the expression of the obstruction $\underline{\calO}$ in terms of tensors on the boundary. 
\ \\

By \eqref{eq-v} and Proposition \ref{coefficients}, we obtain the following theorem:
\begin{thm}\label{L-exp}
Let $\tau\in\calE(1)$ be a normalized projective scale. Then the coefficients $c_j$ and $L$ in the volume expansion \eqref{volexp} are given by the integrals of linear combinations of complete contractions of $h_{\a\b}$, ${\rm Ric}^h_{\a\b}$, $S_\a{}^\b$, $A_{\a\b\g}$ and their covariant derivatives with respect to $\nabla^h$. 
\end{thm}

\subsection{Explicit calculations}
We calculate the expansion of $v$ and present an explicit formula for $L$ when $n=2$. 
Setting $\r=0$ in \eqref{Pab-Ric}, we have 
\begin{equation}\label{Pab-M}
\begin{aligned}
(n-1)P_{\a\b}|_M&={\rm Ric}^\xi_{\a\b}+S_{\a\b}-({\rm tr}S)h_{\a\b} \\
&={\rm Ric}^h_{\a\b}-(\d A)_{\a\b}-A_{\a\mu\nu}A_\b{}^{\mu\nu}+S_{\a\b}-({\rm tr}S)h_{\a\b},
\end{aligned}
\end{equation}
where $(\d A)_{\a\b}:=\nabla^h_\g A_{\a\b}{}^\g$. Then setting $\r=0$ in \eqref{r} gives
\begin{equation}\label{r-M}
\begin{aligned}
r|_M&=\frac{1}{n}\bigl({\rm tr}S-h^{\a\b}P_{\a\b}|_M\bigr) \\
&=\frac{2}{n}{\rm tr}S-\frac{1}{n(n-1)}\bigl(\scal_h-|A|^2\bigr). 
\end{aligned}
\end{equation}
By \eqref{eq-v} we have
$$
\xi v|_M={\rm tr}S-r|_M=\frac{1}{n(n-1)}\bigl(\scal_h-|A|^2\bigr)+\frac{n-2}{n}{\rm tr}S.
$$
Therefore, when $n=2$, $L$ agrees with a nonzero multiple of 
\begin{equation}\label{L2}
4\pi\chi(M)-\int_M |A|^2vol_h
\end{equation}
by the Gauss--Bonnet theorem. 

We will give a formula for $L$ when $n=4$ in a suitably chosen scale in \S \ref{harm-scale}.

\section{Conformal Codazzi structure}
\subsection{Correspondence between projective and conformal tractor bundles}
Let $(N, [\nabla])$ be an $(n+1)$-dimensional locally flat projective manifold and let $M\subset N$ be a strictly convex hypersurface. Then we have two bundles over $M$: the restriction of the projective tractor bundle $\wt\calE^I$ and the conformal tractor bundle for the affine metric, which we denote by $\calE^I$. We will see that there exists a canonical isomorphism between these two bundles.   

First we recall the conformal tractor bundle and conformal tractor connection. For each choice of a conformal  scale $\kappa\in\calE[1]$, the conformal tractor bundle 
is expressed as 
$$
\calE^I\overset{\kappa}{\cong}
\begin{matrix}
\calE[1] \\
\oplus \\
\calE^\a[-1] \\
\oplus \\
\calE[-1]
\end{matrix}.
$$
If we change the scale as $\wh\kappa=e^{-\U}\kappa$, the expression transforms as
\begin{equation}\label{conf-tract-trans}
\begin{pmatrix}
\wh\sigma \\
\wh\mu^\a \\
\wh\varphi
\end{pmatrix}=
\begin{pmatrix}
\sigma \\
\mu^\a+\U^\a\sigma \\
\varphi-\U_\g\mu^\g-\frac{1}{2}\U_\g\U^\g\sigma
\end{pmatrix}.
\end{equation}
The {\it tractor metric} $h_{IJ}\in\calE_{(IJ)}$ is a conformally invariant symmetric form defined by
$$
h_{IJ}U^I{U^\prime}^J=\bh_{\a\b}\mu^\a{\mu^\prime}^\b+\sigma\varphi^\prime+\varphi\sigma^\prime,
$$
where $U^I={}^t(\sigma, \mu^\a, \varphi)$ and ${U^\prime}^I={}^t(\sigma^\prime, {\mu^\prime}^\a, \varphi^\prime)$.
When $n\ge 3$, the {\it conformal tractor connection} is defined by
\begin{equation}\label{conf-tract-conn}
\nabla_\a\begin{pmatrix}
\sigma \\
\mu^\b \\
\varphi
\end{pmatrix}=
\begin{pmatrix}
\nabla^h_\a\sigma-\mu_\a \\
\nabla^h_\a\mu^\b+\d_\a{}^\b\varphi+P^h_{\a}{}^\b\sigma \\
\nabla^h_\a\varphi-P^h_{\a\g}\mu^\g
\end{pmatrix},
\end{equation}
where 
$$
P^h_{\a\b}=\frac{1}{n-2}\Bigl(\Ric^h_{\a\b}-\frac{\scal_h}{2(n-1)}\bh_{\a\b}\Bigl)
$$
is the {\it (conformal) Schouten tensor} of the representative metric $h_{\a\b}=\kappa^{-2}\bh_{\a\b}$. The tractor connection is conformally invariant and preserves the tractor metric. The curvature of the  connection is given by
\begin{equation}\label{conf-tract-curv}
\Omega_{\a\b}{}^K{}_{L}=
\begin{pmatrix}
0 & 0 & 0 \\
2Y_{\a\b}{}^\g & W^h_{\a\b}{}^\g{}_\d & 0 \\
0 & -2Y_{\a\b\d} & 0 
\end{pmatrix},
\end{equation}
where 
$$
Y_{\a\b\g}=\nabla^h_{[\a}P^h_{\b]\g}
$$
is the {\it Cotton--York tensor} and 
$$
W^h_{\a\b\g\d}=R^h_{\a\b\g\d}-2\bigl(\bh_{\g[\a}P^h_{\b]\d}-\bh_{\d[\a}P^h_{\b]\g}\bigr)
$$
is the {\it conformal Weyl tensor}.
When $n=2$, we need an additional structure on $(M, [h])$ in order to define a tractor connection. 
For a fixed conformal scale $\kappa\in\calE[1]$, we set 
\begin{equation}\label{Mobius}
\mathbb{P}_{\a\b}=\mathbb{P}^\kappa_{\a\b}:={\rm tf}\bS_{\a\b}-\frac{1}{2}(\d A)_{\a\b}+\frac{1}{4}\scal_h \bh_{\a\b}\in\calE_{(\a\b)},
\end{equation}
where ${\rm tf}\bS_{\a\b}$ denotes the trace-free part of $\bS_{\a\b}$ with respect to $\bh_{\a\b}$.
Then, by Proposition \ref{affine-transform} and \ref{sym-A-S}, it holds that 
$$
\mathbb{P}_\a{}^\a=\frac{1}{2}\scal_h
$$
and 
$$
\wh{\mathbb{P}}_{\a\b}=\mathbb{P}_{\a\b}-\nabla^h_\a\U_\b+\U_\a\U_\b-\frac{1}{2}\U_\g\U^\g \bh_{\a\b}
$$
for a rescaling $\wh\kappa=e^{-\U}\kappa$. These are the equations described in \cite[Definition 2.2]{R} so $\mathbb{P}_{\a\b}$ defines a {\it M\"obius structure} on $(M, [h])$ in the sense of \cite{R}. (See also \cite{C} for M\"obius structures.) We denote the set $\{ {\mathbb{P}}^\kappa_{\a\b}\ |\ 0<\kappa\in\calE[1]\}$ by $[\mathbb{P}]$. The tensor $\mathbb{P}_{\a\b}$ satisfies $R^h_{\a\b\g\d}=2\bh_{\g[\a}\mathbb{P}_{\b]\d}-2\bh_{\d[\a}\mathbb{P}_{\b]\g}$, so it is a 2-dimensional analogue of the conformal Schouten tensor. We set $P^h_{\a\b}:=\mathbb{P}_{\a\b}$ for a M\"obius surface and define the {\it M\"obius tractor connection} by \eqref{conf-tract-conn}. The curvature is given by \eqref{conf-tract-curv} with $W^h_{\a\b\g\d}=0$.

Let us denote the inclusion map by $\iota:M\rightarrow N$. The isomorphism between $\wt\calE^I|_M$ and $\calE^I$ is given by the following proposition:
\begin{prop}
Let $\tau\in\calE(1)$ be a projective scale and set $J^\prime=\frac{1}{2(n-1)}(\scal_h-|A|^2)\in\calE[-2]$ and $\bxi^i=\tau^{-2}\xi^i\in\calE^i(-2)|_M$, where $\xi^i$ is the affine normal field. Then the bundle isomorphism
$$
\Phi: \calE^I\overset{\tau|_M}{\cong}
\begin{matrix}
\calE[1] \\
\oplus \\
\calE^\a[-1] \\
\oplus \\
\calE[-1]
\end{matrix}
\ni 
\begin{pmatrix}
\sigma \\
\mu^\a \\
\varphi
\end{pmatrix}
\longmapsto
\begin{pmatrix}
\iota_\ast\mu^\a+\sigma\bxi^i \\
\varphi-\frac{1}{n}({\rm tr}\bS-J^\prime)\sigma
\end{pmatrix}
\in 
\begin{matrix}
\calE^i(-1)|_M \\
\oplus \\
\calE(-1)|_M
\end{matrix}
\overset{\tau}{\cong}\wt\calE^I|_M
$$
is independent of $\tau$. Moreover, if $\br\in\calE(2)$ is a Fefferman defining density of $M$, it holds that
$\Phi^\ast \wt g_{IJ}=h_{IJ}$, where $\wt g_{IJ}=D_ID_J\br$ is the ambient metric.
\end{prop}
\begin{proof}
Let $\wh\tau=e^{-\U}\tau$ be another projective scale. By Proposition \ref{affine-transform} and \eqref{conf-tract-trans}, we have
\begin{align*}
\iota_\ast\wh\mu^\a+\wh\sigma\wh\bxi^i &=\iota_\ast\mu^\a+\sigma\iota_\ast\U^\a+\sigma(\bxi^i-\iota_\ast\U^\a) \\
&=\iota_\ast \mu^\a+\sigma\bxi^i, \\
\wh\varphi-\frac{1}{n}({\rm tr}\wh\bS-\wh J^\prime)\wh\sigma
&=\varphi-\U_\g\mu^\g-\frac{1}{2}\U_\g\U^\g\sigma-\frac{1}{n}\Bigl({\rm tr}\bS-J^\prime+n\,\bxi^i\U_i-\frac{n}{2}\U_\g\U^\g\Bigr)\sigma \\
&=\varphi-\frac{1}{n}({\rm tr}\bS-J^\prime)\sigma-\U_i(\iota_\ast\mu+\sigma\bxi^i).
\end{align*}
This agrees with the transformation formula for the projective tractor, so $\Phi$ is invariant under rescaling.
Let $\br$ be a Fefferman defining density of $M$. Using \eqref{g-M} and \eqref{r-M}, we have
\begin{align*}
(\Phi^\ast\wt g_{IJ})U^I{U^\prime}^J&=\bh_{\a\b}\mu^\a{\mu^\prime}^\b+\boldsymbol{r}|_M\sigma\sigma^\prime+\Bigl(\varphi-\frac{1}{n}({\rm tr}\bS-J^\prime)\sigma\Bigr)\sigma^\prime \\
&\quad +\sigma\Bigl(\varphi^\prime-\frac{1}{n}({\rm tr}\bS-J^\prime)\sigma^\prime\Bigr)  \\
&=\bh_{\a\b}\mu^\a{\mu^\prime}^\b+\varphi\sigma^\prime+\sigma\varphi^\prime \\
&=h_{IJ}U^I{U^\prime}^J
\end{align*}
for  $U^I={}^t(\sigma, \mu^\a, \varphi)$ and ${U^\prime}^I={}^t(\sigma^\prime, {\mu^\prime}^\a, \varphi^\prime)$. Thus $\Phi^\ast\wt g_{IJ}=h_{IJ}$ holds.
\end{proof}
\begin{rem}
\rm The correspondence between the projective and conformal tractor bundles is also considered by \v Cap and Gover \cite{CG2} in the context of projective compactification. In our setting, when $\nabla\in[\nabla]$ is a representative connection corresponding to a projective scale $\tau$ and  $\br=\tau^2\r$ is a Fefferman defining density, the connection $\nabla^\prime=\nabla-(2\r)^{-1}d\r$ is projectively compact of order $2$ and the isomorphism described in \cite{CG2} agrees with 
the one given above.
\end{rem}
\subsection{Comparison of connections on the conformal tractor bundle}
We denote by $\nabla^{\rm proj}_\a$ the restriction of the projective tractor connection to $\wt\calE^I|_M$, and set $\bnabla_\a:=\Phi^{-1}\circ\nabla^{\rm proj}_\a\circ\Phi$. A calculation shows that the flat connection 
$\bnabla$ and the conformal (M\"obius) tractor connection $\nabla$ are related as 
\begin{equation}\label{proj-conf}
\bnabla_\a U^J=\nabla_\a U^J-\wt A_{\a K}{}^J U^K, \quad \wt A_{\a K}{}^J=
\begin{pmatrix}
0 & 0 & 0 \\
E_\a{}^\b & A_{\a\g}{}^\b & 0 \\
F_\a & G_{\a\g} & 0
\end{pmatrix},
\end{equation}
where 
\begin{equation}\label{EGF}
\begin{aligned}
E_{\a\b}&=-{\rm tf}\bS_{\a\b}-\frac{1}{n}J^\prime\bh_{\a\b}+P^h_{\a\b} \in\calE_{(\a\b)}, \\
G_{\a\b}&=P_{\a\b}+\frac{1}{n}({\rm tr}\bS-J^\prime)\bh_{\a\b}-P^h_{\a\b}\in\calE_{(\a\b)}, \\
F_{\a}&=\tau^{-2}P_{\a\infty}+\frac{1}{n}\nabla^h_\a({\rm tr}\bS-J^\prime)\in\calE_\a[-2]
\end{aligned}
\end{equation}
in any projective scale $\tau\in\calE(1)$. By \eqref{Pab-Ric}, it holds that $E_\g{}^\g+G_\g{}^\g=0$. 

The following proposition asserts that the flatness of the connection $\bnabla$ determines $E_{\a\b}$, $G_{\a\b}$, $F_\a$ and imposes some equations on $A_{\a\b\g}$:
\begin{prop}\label{GC-prop}
Let $\bnabla$ be a connection on $\calE^I$ of the form \eqref{proj-conf} with some tensors $E_{\a\b}\in\calE_{(\a\b)}$, $G_{\a\b}\in\calE_{(\a\b)}$, $F_\a\in\calE_\a[-2]$ and $A_{\a\b\g}\in\calE_{(\a\b\g)_0}[2]$ in a fixed conformal scale. We assume that $E_\g{}^\g+G_\g{}^\g=0$. When $n=2$, we also assume that ${\rm tf}(E_{\a\b}-G_{\a\b})=0$. Then $\bnabla$ is flat if and only if the following two conditions hold: 
\begin{itemize}
\item[(i)] The tensors $E_{\a\b}$, $G_{\a\b}$ and $F_\a$ satsify
\begin{align}
E_{\a\b}&=-\frac{1}{n}(\d A)_{\a\b}+P^A_{\a\b}, \label{Eab} \\
G_{\a\b}&=-\frac{1}{n}(\d A)_{\a\b}-P^A_{\a\b}, \label{Gab}\\
F_\a&=\frac{1}{n(n-1)}(\d^2 A)_\a+\frac{1}{n-1}A_{\a}{}^{\mu\nu}(P^h_{\mu\nu}-P^A_{\mu\nu}), \label{Fa}
\end{align}
where 
$$
P^A_{\a\b}=
\begin{cases}
\displaystyle\frac{1}{n-2}\Bigl(A_{\a\mu\nu}A_{\b}{}^{\mu\nu}-\frac{|A|^2}{2(n-1)}\bh_{\a\b}\Bigr) & (n\ge 3) \\
\quad \\
\displaystyle\frac{1}{4}|A|^2\bh_{\a\b} & (n=2).
\end{cases}
$$
\item[(ii)] The tensor $A_{\a\b\g}$ satisfies the Gauss equation
\begin{align}
2\,{\rm tf}(A_{\nu\g[\a}A_{\b]\mu}{}^\nu)+W^h_{\a\b\g\mu}&=0 & (n\ge4), \label{Gauss4} \\
\nabla^h_{[\a}P^A_{\b]\g}+\frac{1}{3}A_{\mu\g[\a}(\d A)_{\b]}{}^\mu-Y_{\a\b\g}&=0 & (n=3), \label{Gauss3} \\
\nabla^h_{\a}|A|^2+2A_\a{}^{\mu\nu}(\d A)_{\mu\nu}-8Y_{\a\mu}{}^\mu&=0 & (n=2), \label{Gauss2}
\end{align}
and the Codazzi equation
\begin{align}
\nabla^h_{[\a}A_{\b]\g\mu}-\frac{1}{n}\Bigl(\bh_{\mu[\a}(\d A)_{\b]\g}+\bh_{\g[\a}(\d A)_{\b]\mu}\Bigr)&=0
& (n\ge3), \label{Codazzi3} \\
\nabla^h_{[\a}(\d^2 A)_{\b]}+4P^h_{\mu[\a}(\d A)_{\b]}{}^\mu-2A^{\g\mu}{}_{[\a}\nabla^h_{\b]}P^h_{\g\mu}&=0  & (n=2). \label{Codazzi2}
\end{align}
\end{itemize}
\end{prop}
\begin{proof}
We denote the curvature of $\bnabla_\a$ and $\nabla_\a$ by $\Omega^{\bnabla}_{\a\b M}{}^K$ and $\Omega^{\nabla}_{\a\b M}{}^K$ respectively. Then we can compute as 
\begin{align*}
(1/2)\Omega^{\bnabla}_{\a\b M}{}^K&=-\nabla_{[\a}\wt A_{\b]M}{}^K-\wt A_{[\a|M|}{}^L\wt A_{\b]L}{}^K+(1/2)\Omega^{\nabla}_{\a\b M}{}^K \\
&=\begin{pmatrix}
0 & 0 & 0 \\
\Omega_{\a\b\infty}{}^\g & \Omega_{\a\b\mu}{}^\g & 0 \\
\Omega_{\a\b\infty}{}^0 & \Omega_{\a\b\mu}{}^0 & 0
\end{pmatrix},
\end{align*}
where 
\begin{align}
\Omega_{\a\b\mu}{}^\g&=-\nabla^h_{[\a}A_{\b]\mu}{}^\g-\bh_{\mu[\a}E_{\b]}{}^\g-\d_{[\a}{}^\g G_{\b]\mu}-A_{\nu\mu[\a}A_{\b]}{}^{\g\nu}+\frac{1}{2}W^h_{\a\b}{}^\g{}_\mu, \label{Omega1} \\
\Omega_{\a\b\infty}{}^\g&=-\nabla^h_{[\a}E_{\b]}{}^\g+P^h_{\nu[\a}A_{\b]}{}^{\nu\g}-\d_{[\a}{}^\g F_{\b]}-
E_{\nu[\a}A_{\b]}{}^{\g\nu}+Y_{\a\b}{}^\g, \label{Omega2} \\
\Omega_{\a\b\mu}{}^0&=-\nabla^h_{[\a}G_{\b]\mu}+P^h_{\nu[\a}A_{\b]\mu}{}^\nu-\bh_{\mu[\a}F_{\b]}-
A_{\nu\mu[\a}G_{\b]}{}^\nu-Y_{\a\b\mu}, \label{Omega3} \\
\Omega_{\a\b\infty}{}^0&=-\nabla^h_{[\a}F_{\b]}+P^h_{\nu[\a}E_{\b]}{}^\nu+P^h_{\nu[\a}G_{\b]}{}^\nu-E_{\nu[\a}G_{\b]}{}^\nu. \label{Omega4}
\end{align}
By taking the symmetric part of \eqref{Omega1}, we have 
$$
\Omega_{\a\b(\mu\g)}=-\nabla^h_{[\a}A_{\b]\mu\g}-\frac{1}{2}\bh_{\mu[\a}(E_{\b]\g}+G_{\b]\g})
-\frac{1}{2}\bh_{\g[\a}(E_{\b]\mu}+G_{\b]\mu}).
$$
Using the assumption $E_\g{}^\g+G_\g{}^\g=0$, we see that $\Omega_{\a\b(\mu\g)}=0$ if and only if $E_{\b\g}+G_{\b\g}=-(2/n)(\d A)_{\b\g}$ and $A_{\a\b\g}$ satisfies \eqref{Codazzi3}. Note that the equation \eqref{Codazzi3} always holds when $n=2$ since a tensor $K_{\a\b\mu\g}=K_{[\a\b](\mu\g)}$ is determined by the trace $K_{\a\b}{}^\a{}_\g$. Next we take the skew symmetric part of \eqref{Omega1} to obtain
$$
2\,\Omega_{\a\b[\mu\g]}=-\bh_{\mu[\a}(E_{\b]\g}-G_{\b]\g})+\bh_{\g[\a}(E_{\b]\mu}-G_{\b]\mu})-2A_{\nu\mu[\a}A_{\b]\g}{}^\nu+W^h_{\a\b\g\mu}.
$$
It follows that when $n\ge3$, $\Omega_{\a\b[\mu\g]}=0$ if and only if $E_{\b\g}-G_{\b\g}=2P^A_{\b\g}$ and 
$A_{\a\b\g}$ satisfies \eqref{Gauss4}, but in the case $n=3$ the equation \eqref{Gauss4} becomes vacuous since a tensor with the Weyl curvature symmetry automatically vanishes. When $n=2$, using $W^h_{\a\b\g\mu}=0$ and the fact that a tensor $K_{\a\b\mu\g}$ with the Riemannian curvature symmetry is determined by its trace $K_{\a\b}{}^{\a\b}$, we see that $\Omega_{\a\b[\mu\g]}=0$ if and only if $E_\g{}^\g-G_\g{}^\g=|A|^2$. Since we have assumed that ${\rm tf}(E_{\b\g}-G_{\b\g})=0$, this is equivalent to $E_{\b\g}-G_{\b\g}=
(1/2)|A|^2\bh_{\b\g}$. 

Thus the condition $\Omega_{\a\b\mu\g}=0$ determines $E_{\a\b}$ and $G_{\a\b}$ as in \eqref{Eab} and \eqref{Gab}, and imposes the equations \eqref{Gauss4} and \eqref{Codazzi3}. 

By \eqref{Eab}, \eqref{Gab}, \eqref{Omega2} and \eqref{Omega3}, we have
\begin{align*}
\Omega_{\g\b\infty}{}^\g+\Omega_{\g\b}{}^{\g 0}&=
\frac{1}{n}(\d^2A)_\b+A_\b{}^{\g\mu}(P^h_{\mu\g}-P^A_{\mu\g})-(n-1)F_\b, \\
\Omega_{\g\b\infty}{}^\g-\Omega_{\g\b}{}^{\g 0}&=\nabla^h_\b P^A_\g{}^\g-\nabla^h_\g P^A_\b{}^\g
+\frac{1}{n}A_{\b}{}^{\mu\g}(\d A)_{\mu\g}-2Y_{\b\g}{}^\g.
\end{align*}
When $n\ge3$, as $Y_{\b\g}{}^\g=0$, we see that $\Omega_{\g\b\infty}{}^\g-\Omega_{\g\b}{}^{\g 0}=0$ follows from \eqref{Codazzi3}, so the condition 
$\Omega_{\g\b\infty}{}^\g=\Omega_{\g\b}{}^{\g 0}=0$ is equivalent to the equation \eqref{Fa}. On  the other hand, when $n=2$ it holds that
$$
\Omega_{\g\b\infty}{}^\g-\Omega_{\g\b}{}^{\g 0}=\frac{1}{4}\nabla^h_\b|A|^2+\frac{1}{2}A_\b{}^{\mu\g}(\d A)_{\mu\g}-2Y_{\b\g}{}^\g.
$$
Therefore $\Omega_{\g\b\infty}{}^\g=\Omega_{\g\b}{}^{\g 0}=0$ is equivalent to \eqref{Fa} and \eqref{Gauss2}.

Let us  consider the Bianchi identity 
$$
\nabla_{[\nu}\Omega^{\bnabla}_{\a\b]M}{}^K-\wt A_{[\nu|L|}{}^K\Omega^{\bnabla}_{\a\b]M}{}^L+\Omega^{\bnabla}_{[\nu\a|L|}{}^K\wt A_{\b]M}{}^L=0.
$$
When $\Omega_{\a\b\mu\g}=0$, the above equation for $(M, K)=(\mu, \g)$ and $(M, K)=(\infty, \g)$ gives
\begin{align}
\bh_{\mu[\nu}\Omega_{\a\b]\infty}{}^\g+\d_{[\nu}{}^\g\Omega_{\a\b]\mu}{}^0&=0, \label{Bianchi1} \\
\nabla^h_{[\nu}\Omega_{\a\b]\infty}{}^\g+\d_{[\nu}{}^\g\Omega_{\a\b]\infty}{}^0-A_\g{}_{\d[\nu}\Omega_{\a\b]\infty}{}^\d&=0. \label{Bianchi2}
\end{align}
Taking the traces of \eqref{Bianchi1} under the condition $\Omega_{\g\b\infty}{}^\g=\Omega_{\g\b}{}^{\g 0}=0$, we have
$$
(n-2)\Omega_{\a\b\infty}{}^\g+\Omega_{\a\b}{}^{\g0}=0, \quad
\Omega_{\a\b\infty}{}^\g+(n-2)\Omega_{\a\b}{}^{\g0}=0.
$$
It follows that $\Omega_{\a\b\infty}{}^\g=\Omega_{\a\b}{}^{\g0}=0$ when $n\neq3$ and that $\Omega_{\a\b\infty}{}^\g+\Omega_{\a\b}{}^{\g0}=0$ when $n=3$. Then, when $n\neq3$ the equation \eqref{Bianchi2} implies 
$(n-2)\Omega_{\a\b\infty}{}^0=0$ so that $\bnabla$ is flat when $n\ge4$. When $n=2$, it follows from \eqref{Omega4}, \eqref{Eab}, \eqref{Gab}, and \eqref{Fa} that $\Omega_{\a\b\infty}{}^0=0$ is equivalent to the equation \eqref{Codazzi2}. When $n=3$, the equation $\Omega_{\a\b\infty}{}^\g=\Omega_{\a\b}{}^{\g0}=0$ holds if and only if  $\Omega_{\a\b\infty}{}^\g-\Omega_{\a\b}{}^{\g0}=0$, which is equivalent to \eqref{Gauss3}, and if the equation holds, $\Omega_{\a\b\infty}{}^0=0$ follows from \eqref{Bianchi2}. Thus we complete the proof.
\end{proof}

\begin{rem}
\rm When $n=2$ there are surfaces which have the same conformal structure and the same Fubini--Pick form but  have different M\"obius structures, so we need the condition ${\rm tf}(E_{\a\b}-G_{\a\b})=0$ to fix a M\"obius structure. In the case of $\bnabla_\a=\Phi^{-1}\circ\nabla^{\rm proj}_\a\circ\Phi$, we can directly prove the equations \eqref{Eab} and \eqref{Gab} by using \eqref{Pab-Ric}, \eqref{Mobius} and the fact that 
$\Ric_{\a\b}=(1/2)\scal_h\bh_{\a\b}$, $A_{\a\mu\nu}A_{\b}{}^{\mu\nu}=(1/2)|A|^2\bh_{\a\b}$.
\end{rem}

Comparing \eqref{EGF} with \eqref{Eab} and \eqref{Gab}, we obtain the following identities:

\begin{align}
{\rm tf}\bS_{\a\b}&=\frac{1}{n}(\d A)_{\a\b}+{\rm tf}(P^h_{\a\b}-P^A_{\a\b}), \label{tfS} \\
{\rm tf}P_{\a\b}&=-\frac{1}{n}(\d A)_{\a\b}+{\rm tf}(P^h_{\a\b}-P^A_{\a\b}). \label{tfP}
\end{align}

Next we compare the Levi--Civita connection $\wt\nabla$ of the ambient metric and the conformal tractor connection. We decompose $\wt A_{\a J}{}^K$ into the symmetric and the skew symmetric parts:
$$
\wt A_{\a J}{}^K={\wt A}^\prime_{\a J}{}^K+{\wt A}^\prime{}^\prime_{\a J}{}^K, 
\quad {\wt A}^\prime_{\a JK}=\wt A_{\a (JK)}, \ {\wt A}^\prime{}^\prime_{\a JK}=\wt A_{\a [JK]}.
$$
Then we have the following proposition:
\begin{prop}
Let $\br$ be a Fefferman defining density and let $\wt\nabla_I$ be the Levi--Civita connection of the ambient metric $\wt g_{IJ}=D_ID_J\br$. Then the restriction  $\wt\nabla_\a$ is well-defined and satisfies
$$
\Phi^{-1}\circ\wt\nabla_\a\circ\Phi=\nabla_\a-{\wt A}^\prime{}^\prime_{\a J}{}^K,
$$
where $\nabla_\a$ is the conformal (M\"obius) tractor connection. In particular, $\Phi^{-1}\circ\wt\nabla_\a\circ\Phi=\nabla_\a$ if and only if $A_{\a\b\g}=0$.
\end{prop}
\begin{proof}
Since $D_I$ is flat, the connection $\wt\nabla_I$ satisfies $\wt\nabla_I=D_I+\wt\Gamma_{IJ}{}^K$, where $\wt\Gamma_{IJK}=(1/2)D_ID_JD_K\br\in\wt\calE_{(IJK)}(-1)$. It follows that 
$T^I\wt\nabla_I\nu^J=0$ for any $\nu^I\in\wt\calE^I$, so one can define $\wt\nabla_i\nu^J$ invariantly and hence $\wt\nabla_\a\nu^J$ is well-defined on $\wt\calE^I|_M$. The projecting part of 
$\wt\Gamma_{IJK}$ is computed as
$$
2\wt\Gamma_{ijk}=\nabla_i\nabla_j\nabla_k\br+2\br\nabla_iP_{jk}+2P_{jk}\nabla_i\br+P_{ij}\nabla_k\br
+P_{ik}\nabla_j\br.
$$
We fix a projective scale $\tau\in\calE(1)$, and work in an adapted frame $\{e_\infty=\xi, e_\a\}$. 
Using \eqref{conn-adapted}, \eqref{r-M}, and \eqref{EGF}, we can show that the identities 
$$
\wt\Gamma_{\a\b\g}=A_{\a\b\g},\quad \wt\Gamma_{\a\b\infty}=\frac{1}{2}(E_{\a\b}+G_{\a\b}), \quad
\wt\Gamma_{\a\infty\infty}=F_\a
$$
hold on the hypersurface $M$. It follows that 
$$
\Phi^{\ast}\wt\Gamma_{\a J}{}^K=
\begin{pmatrix}
0 & 0 & 0 \\
\frac{1}{2}(E_\a{}^\g+G_\a{}^\g) & A_{\a\b}{}^\g & 0 \\
F_\a & \frac{1}{2}(E_{\a\b}+G_{\a\b}) & 0
\end{pmatrix}
={\wt A}^\prime{}_{\a J}{}^K.
$$
Thus we have $\Phi^{-1}\circ\wt\nabla_\a\circ\Phi=\bnabla_\a+{\wt A}^\prime{}_{\a J}{}^K=\nabla_\a-{\wt A}^\prime{}^\prime_{\a J}{}^K$. 

By \eqref{Eab} and \eqref{Gab}, it holds that
$$
{\wt A}^\prime{}^\prime_{\a JK}=
\begin{pmatrix}
0 & 0 & 0 \\
P^A_{\a\g} & 0 & 0 \\
0 & -P^A_{\a\b} & 0
\end{pmatrix}.
$$
Since ${\rm tr}P^A=(1/2(n-1))|A|^2$, the condition ${\wt A}^\prime{}^\prime_{\a JK}=0$ is equivalent to $A_{\a\b\g}=0$.
\end{proof}

\subsection{Conformal Codazzi manifold}
As an intrinsic geometric structure of a hypersurface in a locally flat projective manifold, we define the {\it conformal Codazzi structure} as follows:
\begin{dfn}\label{conf-Codazzi}
Let $M$ be a $C^\infty$-manifold of dimension $n$. When $n\ge3$, a conformal Codazzi structure is a pair $([h], A_{\a\b\g})$, where $[h]$ is a conformal structure and $A_{\a\b\g}\in\calE_{(\a\b\g)_0}[2]$ satisfies the Gauss--Codazzi equations \eqref{Gauss4}, \eqref{Codazzi3} for $n\ge4$ and \eqref{Gauss3}, \eqref{Codazzi3} for $n=3$. When $n=2$, a conformal Codazzi structure is a triad $([h], [\mathbb{P}], A_{\a\b\g})$, where 
$[h]$ is a conformal structure, $[\mathbb{P}]$ is a M\"obius structure and $A_{\a\b\g}\in\calE_{(\a\b\g)_0}[2]$ satisfies the Gauss--Codazzi equations \eqref{Gauss2}, \eqref{Codazzi2}.
\end{dfn}

We will call $(M, [h], A_{\a\b\g})$ a conformal Codazzi manifold, suppressing the M\"obius structure $[\mathbb{P}]$ when $n=2$. On the conformal tractor bundle over a conformal Codazzi manifold, there is a flat connection $\bnabla$ defined by \eqref{proj-conf} with \eqref{Eab}, \eqref{Gab}, and \eqref{Fa}. We call the connection $\bnabla$ 
the {\it conformal Codazzi tractor connection}.

We observe that when $A_{\a\b\g}$ satisfies the Gauss--Codazzi equations, then $A^\ast_{\a\b\g}:=-A_{\a\b\g}$ also satisfies the equations. We call $([h], A^\ast_{\a\b\g})$ the {\it dual conformal Codazzi structure}. 
The corresponding conformal Codazzi tractor connection $\bnabla^\ast$ is given by 
$$
\bnabla^\ast_\a=\nabla_\a-{\wt A}^\prime_{\a J}{}^K+{\wt A}^\prime{}^\prime_{\a J}{}^K,
$$
and called the {\it dual conformal Codazzi tractor connection}. It is the dual connection of $\bnabla$ with respect to the conformal tractor metric $h_{IJ}$ in the sense that 
$$
\nabla^h_\a(h_{IJ}U^IV^J)=h_{IJ}(\bnabla_\a U^I)V^J+h_{IJ}U^I(\bnabla^\ast_\a V^J)
$$
holds for $U^I, V^I\in\calE^I$.

Recall from \cite{BEG} that the conformal tractor connection is defined via the prolongation of the conformal-to-Einstein equation ${\rm tf}(\nabla^h_\a\nabla^h_\b\sigma+P^h_{\a\b}\sigma)=0$ for $\sigma\in\calE[1]$.
We will see that when $n\ge3$ the conformal Codazzi tractor connection is also defined via the prolongation of a conformally invariant differential equation. We set   
$$
{\mathcal P}_{\a\b}:=-\frac{1}{n}(\d A)_{\a\b}+{\rm tf}(P^h_{\a\b}-P^A_{\a\b})\in\calE_{(\a\b)_0}.
$$
When $M$ is a hypersurface of a locally flat projective manifold, the tensor $\mathcal P_{\a\b}$ agrees with 
${\rm tf}P_{\a\b}$ by \eqref{tfS}. We consider a conformally invariant differential equation
\begin{equation}\label{diff-eq}
{\rm tf}(\nabla^h_\a\nabla^h_\b\sigma)+A_{\a\b}{}^\g\nabla^h_\g\sigma+\mathcal P_{\a\b}\sigma=0
\end{equation}
for a density $\sigma\in\calE[1]$. Suppose that $\sigma$ is a solution to \eqref{diff-eq} and set $\mu_\a=\nabla^h_\a\sigma$. Then there exists a density $\varphi\in\calE[-1]$ such that
\begin{equation}\label{nabla-mu}
\nabla^h_\a\mu_\b+A_{\a\b}{}^\g \mu_\g+\varphi\bh_{\a\b}+\Bigl(-\frac{1}{n}(\d A)_{\a\b}+P^h_{\a\b}-P^A_{\a\b}\Bigr)\sigma=0.
\end{equation}
Thus the $2$-jet of $\sigma$ at a point $x\in M$ is determined by $(\sigma_x, (\mu_\a)_x, \varphi_x)$. If we change the conformal scale by $\U$, the triad ${}^t(\sigma, \mu^\a, \varphi)$ transforms as \eqref{conf-tract-trans}. Therefore we have the following isomorphism:
$$
\calE_I\cong\{j^2\sigma\in J^2\calE[1]\ |\ {\rm tf}(\nabla^h_\a\nabla^h_\b\sigma)+A_{\a\b}{}^\g\nabla^h_\g\sigma+\mathcal P_{\a\b}\sigma=0\},
$$
where $j^2\sigma$ denotes the 2-jet of $\sigma$ and $J^2\calE[1]$ denotes the jet bundle. Taking the divergence of \eqref{nabla-mu} and using \eqref{Codazzi3}, we have
$$
\nabla^h_\a\varphi+(E_{\a\g}-P^h_{\a\g})\mu^\g+F_\a\sigma=0.
$$
It follows that the linear connection on $\calE_I$ defined by the prolongation of the equation \eqref{diff-eq} coincides with $\bnabla$. Thus, $\sigma\in\calE[1]$ is a solution to \eqref{diff-eq} if and only if there exist 
$\mu_\a\in\calE_\a[1]$ and $\varphi\in\calE[-1]$ such that
$$
\bnabla_\a
\begin{pmatrix}
\sigma \\
\mu_\b \\
\varphi
\end{pmatrix}=0.
$$

We shall consider the dual of this construction. We set
$$
{\mathcal S}_{\a\b}:=\frac{1}{n}(\d A)_{\a\b}+{\rm tf}(P^h_{\a\b}-P^A_{\a\b})\in\calE_{(\a\b)_0}.
$$
The tensor ${\mathcal S}_{\a\b}$ is the dual of ${\mathcal P}_{\a\b}$ and agrees with ${\rm tf}\bS_{\a\b}$ by \eqref{tfS} when $M$ is a hypersurface of a locally flat projective manifold. The dual conformal Codazzi tractor connection $\bnabla^\ast$ is given by the prolongation of the equation dual to \eqref{diff-eq}:
\begin{equation}\label{diff-eq-dual}
{\rm tf}(\nabla^h_\a\nabla^h_\b\sigma)-A_{\a\b}{}^\g\nabla^h_\g\sigma+\mathcal S_{\a\b}\sigma=0.
\end{equation}

A conformal scale $\sigma\in\calE[1]$ is a solution to \eqref{diff-eq} or \eqref{diff-eq-dual} if and only if 
it holds in the scale $\sigma$ that $\mathcal P_{\a\b}=0$ or $\mathcal S_{\a\b}=0$ respectively. Since $\bnabla$ and $\bnabla^\ast$ are flat, such a scale always exists locally. Thus we obtain the following proposition:
\begin{prop}
Let $(M, [h], A_{\a\b\g})$ be a conformal Codazzi manifold of dimension $n\ge3$. Then there exists a local conformal scale for which $\mathcal P_{\a\b}=0$ holds, and there also exists one for which $\mathcal S_{\a\b}=0$.
\end{prop}

\subsection{Projective Bonnet theorem}
We will prove that the conformal Codazzi structure recovers the local immersion to the projective space and hence it has complete information on the local geometric structure of a hypersurface in a locally flat projective manifold. 

\begin{thm}\label{proj bonnet}
Let $(M, [h], A_{\a\b\g})$ be a conformal Codazzi manifold of dimension $n\ge2$. Then for any point $p_0\in M$ there exists a neighborhood $U\subset M$ and an immersion $f:U\rightarrow \R\mathbb P^{n+1}$ such that the induced conformal Codazzi structure agrees with $([h], A_{\a\b\g})$. 
\end{thm}
\begin{proof}
We denote by $\mathcal T$ the conformal tractor bundle over $M$. Let $\Psi\in\wedge^{n+2}{\mathcal T}^\ast$ be the volume form of the tractor metric. Then since $\wt A_{\a J}{}^J=0$, the conformal Codazzi tractor connection $\bnabla$ preserves $\Psi$. It follows from the flatness of $\bnabla$ that there exist a neighborhood $U$ of any point $p_0\in M$ and a frame $\{e_I\}$ for $\mathcal T|_U$ which is parallel with respect to $\bnabla$ and unimodular with respect to $\Psi$. We write $\phi:\mathcal T|_U\rightarrow U\times\R^{n+2}$ for the trivialization via the frame $\{e_I\}$. Then we have $\phi(\bnabla_X V)=X\phi(V)$ for $X\in\Gamma(TU)$ and $V\in\Gamma(\mathcal T|_U)$. Since $\mathcal T$ contains $\calE[-1]$ as a line subbundle, we can define the map $f: U\rightarrow \R\mathbb P^{n+1}$ by $f(p):=[\phi(\calE[-1]_p)]$.
Let $\mathring{\mathcal T}$ and $\mathring{\calE}(-1)$ be respectively the projective tractor bundle over $\R\mathbb P^{n+1}$ and the projective density bundle of weight $-1$ over $\R\mathbb P^{n+1}$. Then $\mathring{\mathcal T}$ is the trivial bundle and $\mathring{\calE}(-1)$ coincides with the tautological line bundle. Moreover the conformal Codazzi tractor connection on $\mathring{\mathcal T}$, which we denote by $\mathring{\bnabla}$, is the canonical flat connection on $\R\mathbb P^{n+1}\times \R^{n+2}$. 
By the definition of $f$ it holds that 
$$
f^\ast\mathring{\mathcal T}=\mathcal T|_U, \quad f^\ast\mathring{\bnabla}=\bnabla, \quad f^\ast
\mathring{\calE}(-1)=\calE[-1].
$$
Let $(x^0,\dots, x^{n+1})$ be the homogeneous coordinates on $\R\mathbb P^{n+1}$. As we may assume that $f(U)\subset \{x^0\neq0\}$, we identify $f$ with a map to the affine space $\R^{n+1}\cong\{x^0=1\}$.
The hyperplane $\{x^0=1\}$ in $\R^{n+2}$ defines a section of $\mathring{\calE}(-1)$ and hence defines a projective scale $\mathring\tau\in\mathring{\calE}(1)$. We decompose $\mathcal T|_U$ by the conformal scale $\tau:=f^\ast\mathring\tau$ as 
$$
\mathcal T|_U=\calE\oplus TU\oplus\calE.
$$
Setting $T:={}^t(0, 0, 1)\in\Gamma(\mathcal T|_U)$, we have $f=\phi(T)$ and hence
$$
f_\ast X=X\phi(T)=\phi(\bnabla_X T)=\phi(X)
$$
for any $X\in\Gamma(TU)$. Therefore, $f_\ast=\phi:TU\rightarrow\R^{n+1}$ and $f$ is an immersion. 

Let us consider the conformal Codazzi structure induced on the hypersurface $f(U)\subset\R^{n+1}$.
We set $Y:={}^t(1, 0, 0)\in\mathcal T|_U$. Let $\pi:\mathring{\mathcal T}|_{\R^{n+1}}\rightarrow T\R^{n+1}$
be the projection along $\phi(T)$ and set $\xi:=\pi(\phi(Y))\in\Gamma(f(U), T\R^{n+1})$. We will see that 
$\xi$ is identical to the affine normal field on $f(U)$ for the projective scale $\mathring\tau$. First, let 
$\mathring h$ be the affine second fundamental form on $f(U)$ with respect to $\xi$. We note that the representative connection $\nabla^{\mathring\tau}$ for the scale $\mathring\tau$ agrees with the canonical flat connection on $\R^{n+1}$. For $X, Z\in\Gamma(TU)$, we have 
$$
\nabla^{\mathring\tau}_{f_\ast X}f_\ast Z=\mathring{\bnabla}_{f_\ast X}f_\ast Z
=f_\ast\bnabla_X Z=\phi(\bnabla_X Z),
$$
where $\bnabla_X Z$ is of the form
\begin{equation}\label{nabla-X-Z}
\bnabla_X Z=
\begin{pmatrix}
-h_{\a\b}X^\a Z^\b \\
X^\b(\nabla^h_\b Z^\a-A_{\b\g}{}^\a Z^\g) \\
\ast
\end{pmatrix}
\end{equation}
with $h_{\a\b}:=\tau^{-2}\bh_{\a\b}$. Since the $\xi$-component of $\phi(\bnabla_X Z)$ equals the $Y$-component of $\bnabla_X Z$, it follows that $f^\ast \mathring h=h$. We set $\mathring\Psi:=dx^0\wedge\cdots\wedge dx^{n+1}\in\wedge^{n+2}{\mathring{\mathcal T}}^\ast$ so that $\phi^\ast\mathring\Psi=\Psi$ holds. Noting that ${\mathring\tau}^{-(n+2)}=dx^1\wedge\cdots\wedge dx^{n+1}$, $\phi(T)=\partial/\partial x^0+x^i\partial/\partial x^i$, and $\phi(Y)\equiv\xi$ mod $\phi(T)$, we have
\begin{align*}
f^\ast(\xi\lrcorner\,{\mathring\tau}^{-(n+2)})|_{Tf(U)}
&=\phi^\ast(\phi(Y)\lrcorner\,\phi(T)\lrcorner\,\mathring\Psi)|_{Tf(U)} \\
&=(Y\lrcorner\, T\lrcorner\,\Psi)|_{TU} \\
&=vol_h \\
&=f^\ast vol_{\mathring h}.
\end{align*}
Thus $\xi\lrcorner\,{\mathring\tau}^{-(n+2)})|_{Tf(U)}=vol_{\mathring h}$ holds. Next, for any $X\in TU$ we have
\begin{align*}
\nabla^{\mathring\tau}_{f_\ast X}\xi&\equiv\pi({\mathring\bnabla}_{\phi(X)}\phi(Y)) \\
&=\pi(\phi(\bnabla_X Y))\quad {\rm mod}\ Tf(U),
\end{align*}
where $\bnabla_X Y$ is of the form
\begin{equation}\label{nabla-X-Y}
\bnabla_X Y=
\begin{pmatrix}
0 \\
X^\b(P^h_\b{}^\a+\frac{1}{n}(\d A)_\b{}^\a-P^A_{\b}{}^\a) \\
\ast
\end{pmatrix}.
\end{equation}
The $\xi$-component of $\nabla^{\mathring\tau}_{f_\ast X}\xi$ equals the $Y$-component of $\bnabla_X Y$, so it vanishes. Thus we have proved that $\xi$ is the affine normal field. 

We already proved that $f^\ast \mathring h=h$. By equation \eqref{nabla-X-Z}, we can see that
the induced connection on $U$ is given by $\nabla^h_\a-A_{\a\b}{}^\g$. Therefore, the Fubini--Pick form is equal to $A_{\a\b\g}$. When $n=2$, it follows from equation \eqref{nabla-X-Y} that the affine shape operator on $U$ is given by $S_{\a\b}=\mathbb P_{\a\b}+(1/2)(\d A)_{\a\b}-(1/4)|A|^2 h_{\a\b}$, and hence  the induced M\"obius structure is ${\rm tf}S_{\a\b}-(1/2)(\d A)_{\a\b}+(1/4)\scal_h h_{\a\b}=\mathbb P_{\a\b}$. Thus we 
complete the proof.
\end{proof}

\subsection{Harmonic scales}\label{harm-scale}
Let $M$ be the boundary of a strictly convex domain in a locally flat projective manifold $N$ of dimension $n+1$. Then the coefficient $L$ of the logarithmic term in the volume expansion of the Blaschke metric gives a projective invariant of $M$ when $n$ is even. We will show that in fact it is a conformal invariant of the conformal Codazzi manifold $(M, [h], A_{\a\b\g})$. Recall that ${\rm tr}\bS$ depends on the 1-jet of the projective scale along the boundary, and cannot be expressed in terms of the conformal Codazzi structure in general. Thus we need to introduce a normalization condition on the 1-jet of a projective scale. 
Let $\wt\Delta=-{\wt g}^{IJ}\wt\nabla_I\wt\nabla_J$ be the Laplacian of the ambient metric 
$\wt g_{IJ}=D_ID_J\br$.
\begin{dfn}
A projective scale $\tau\in\calE(1)$ is called a harmonic scale if it satisfies $\wt\Delta\tau=O(\br)$.
\end{dfn}
Note that the definition does not depend on the choice of $\br$. The following proposition shows that the harmonicity normalizes the transverse derivative of the scale: 
\begin{prop}
{\rm (i)} For any conformal scale $\kappa\in\calE[1]$, there exists a harmonic scale $\tau\in\calE(1)$, unique modulo $O(\br^2)$, such that $\tau|_M=\kappa$.  \\
{\rm (ii)} When $\tau\in\calE(1)$ is a harmonic scale, $\wh\tau=e^{-\U}\tau$ is harmonic if and only if 
$$
\bxi \U=\frac{1}{n}(\Delta_h\U+|d\U|_{\bh}^2),
$$
where $\bxi$ is the (weighted) affine normal field for $\tau$ and $\Delta_h=-\bh^{\a\b}\nabla^h_\a\nabla^h_\b$. \\
{\rm (iii)} A projective scale $\tau\in\calE(1)$ is harmonic if and only if 
\begin{equation}\label{trS-harmonic}
{\rm tr}\bS=\frac{1}{n-1}(\scal_h-|A|^2)
\end{equation}
holds.
\end{prop}
\begin{proof}
{\rm (i)} Let $\mathring\tau$ be an arbitrary extension of $\kappa$ and set $\tau=\mathring\tau+\br\phi$ for
a density $\phi\in\calE(-1)$. Then using $\wt\Delta\br=-(n+2)$ and $\wt g^{IJ}D_J\br=T^I$, we have
$$
\wt\Delta\tau=\wt\Delta\mathring\tau-n\phi+O(\br).
$$
Thus there is a $\phi$ unique modulo $O(\br)$ such that $\wt\Delta\tau=O(\br)$. 

{\rm (ii)} By the Monge--Amp\`{e}re equation $\det \wt g_{IJ}=-1+O(\br^{n/2+1})$, the Christoffel symbols of 
$\wt\nabla$ satisfy $\wt g^{IJ}\wt\Gamma_{IJ}{}^K=O(\br^{n/2})$. Thus we have
\begin{align*}
\wt\Delta\wh\tau&=-\wt g^{IJ}D_ID_J\wh\tau+O(\br) \\
&=e^{-\U}\wt g^{IJ}(\tau D_ID_J\U-\tau D_I\U D_J\U+2D_I\U D_J\tau)+O(\br).
\end{align*}
In an adapted frame $\{e_\infty=\xi, e_\a\}$ for the scale $\tau$, it holds at the boundary that
\begin{align*}
D_I\tau&=
\begin{pmatrix}
\tau \\
0 \\
0
\end{pmatrix},\quad
D_J\U=
\begin{pmatrix}
0 \\
\xi\U \\
\nabla^h_\b\U
\end{pmatrix},  \\
D_ID_J\U&=
\begin{pmatrix}
0 & -\xi\U & -\nabla^h_\b\U \\
-\xi\U & \ast & \ast \\
-\nabla^h_\a\U & \ast & \nabla^h_\a\nabla^h_\b\U+A_{\a\b}{}^\g\nabla^h_\g\U+(\bxi\U)\bh_{\a\b}
\end{pmatrix}. 
\end{align*}
Therefore using \eqref{g-M} we have 
$$
\wt\Delta\wh\tau=e^{-\U}\tau(n\,\bxi\U-\Delta_h\U-|d\U|_{\bh}^2)+O(\br),
$$
which proves the assertion. 

{\rm (iii)} In an adapted frame for the scale $\tau$, it holds at the boundary that
$$
D_ID_J\tau=
\tau\begin{pmatrix}
0 & 0 & 0 \\
0 & P_{\infty\infty} & P_{\b\infty} \\
0 & P_{\a\infty} & P_{\a\b}
\end{pmatrix}.
$$
Thus, by \eqref{Pab-Ric}, we have
$$
\wt\Delta\tau|_M=-\tau\bh^{\a\b}P_{\a\b}=\tau\Bigl({\rm tr}\bS-\frac{1}{n-1}(\scal_h-|A|^2)\Bigr),
$$
from which (iii) follows.
\end{proof}

By the above proposition, a conformal scale on $M$ determines the $1$-jet of a harmonic scale and we can extend it to a normalized projective scale. Thus, it follows from Theorem \ref{L-exp} that $L$ gives a conformal invariant of $M$:
\begin{thm}\label{L-codazzi}
Let $M$ be the boundary of a strictly convex domain in a locally flat projective manifold, and let $\kappa\in\calE[1]$ be a conformal scale on $M$. Then the coefficient of the logarithmic term in the volume expansion of the Blaschke metric can be written in the form
$$
L=\int_M F\,vol_h,
$$
where $F$ is a linear combination of complete contractions of $h_{\a\b}, {\rm Ric}^h_{\a\b}, A_{\a\b\g}$ and their covariant derivatives with respect to $\nabla^h$. Moreover the integral is independent of the choice of $\kappa$. 
\end{thm}

We shall compute the conformal invariant $L$ for $n=4$. Let $\tau\in\calE(1)$ be a harmonic normalized 
scale. By \eqref{eq-v}, $L$ is given by the integral of 
$$
\xi^2 v|_M=(\xi{\rm tr}S-\xi r+({\rm tr S}-r)^2)|_M
$$
over $M$. Since $\tau$ is harmonic, we have 
$$
r|_M=\frac{1}{n}{\rm tr}S=\frac{1}{n(n-1)}(\scal_h-|A|^2).
$$
By differentiating \eqref{r} and using \eqref{xi-h} and \eqref{xi-P}, we have
$$
\xi r|_M=\frac{1}{2}(|P|^2-|S|^2)+({\rm div}),
$$
where $({\rm div})$ denotes divergence terms. The equation \eqref{xi-S} gives
$$
\xi{\rm tr}S|_M=-|S|^2-\frac{1}{4}({\rm tr}S)^2+({\rm div}).
$$
Then it follows from \eqref{tfS} and \eqref{tfP} that 
\begin{align*}
\xi^2 v|_M&=-\frac{1}{2}(|{\rm tf}S|^2+|{\rm tf}P|^2)+\frac{3}{16}({\rm tr}S)^2+({\rm div}) \\
&=\frac{1}{8}\Bigl(|W^h|^2-2|{\rm Ric}^h|^2+\frac{2}{3}\scal_h^2\Bigr)-\frac{1}{8}|W^h|^2 -\frac{1}{4}|A^2|^2 +\frac{1}{12}|A|^4 \\
&\quad-\frac{1}{16}|\d A|^2+\frac{1}{6}|A|^2\scal_h+\frac{1}{2}\langle {\rm Ric}^h, A^2\rangle+({\rm div}),
\end{align*}
where the tensor $A^2$ is defined by $(A^2)_{\a\b}=A_{\a\mu\nu}A_\b{}^{\mu\nu}$. Thus by the Gauss--Bonnet theorem, we obtain
\begin{equation}\label{L4}
\begin{aligned}
L=4\pi^2\chi(M)&-\int_M \Bigl(\frac{1}{8}|W^h|^2 +\frac{1}{4}|A^2|^2 -\frac{1}{12}|A|^4 \Bigr)vol_h \\
&-\int_M\Bigl(\frac{1}{6}|A|^2\scal_h-\frac{1}{2}\langle {\rm Ric}^h, A^2\rangle+\frac{1}{16}|\d A|^2\Bigl)vol_h.
\end{aligned}
\end{equation}


\begin{thebibliography}{ABC}
\bibitem[Al]{Al}P. Albin,
{\it Renormalizing curvature integrals on Poincar\'e-Einstein manifolds,}
Adv. Math. {\bf 221} (2009), 140--169.

\bibitem[An]{An}M. Anderson,
{\it $L^2$-curvature and renormalization of AHE metrics on 4-manifolds,}
Math. Res. Lett. {\bf 8}, (2001), 171--188.

\bibitem[BEG]{BEG}T. N. Bailey, M. G. Eastwood, and A. R. Gover,
{\it Thomas' structure bundles for conformal, projective and related structures,}
 Rocky Mountain J. {\bf 24} (1994), 1191--1217.

\bibitem[BC]{BC}F. E. Burstall and D. M. J. Calderbank, 
{\it Submanifold geometry in generalized flag manifolds,}
in Winter School in Geometry and Physics (Srni, 2003), Rend. del Circ. mat. di Palermo {\bf 72} (2004), 13--41.

\bibitem[C]{C}D. M. J. Calderbank,
{\it M\"obius structures and two-dimensional Einstein-Weyl geometry,}
J. Reine Angew. Math. {\bf 504} (1998), 37--53.




\bibitem[\v CG1]{CG1}A. \v Cap and A. R. Gover,
{\it Tractor calculi for parabolic geometries,}
Trans. Amer. Math. Soc. {\bf 354} (2002), no.4, 1511--1548. 

\bibitem[\v CG2]{CG2}A. \v Cap and A. R. Gover,
{\it Projective compactness and conformal boundaries,}
Math. Ann. {\bf 366} no. 3--4 (2016), 1587--1620.

\bibitem[\v CS]{CS}A. \v Cap and J. Slov\'ak,
{\it Parabolic geometries I: Background and General theory,}
Mathematical Surveys and Monographs. American Mathematical Society, Providence, RI, 2009. 

\bibitem[\v CSS]{CSS}A. \v Cap, A. Slov\'{a}k and V. Sou\v cek, 
{\it Bernstein--Gelfund--Gelfund sequence,}
Ann. of Math. (2) {\bf 154} (2001), 97--113.

\bibitem[CaY]{CaY}J. Case and P. Yang,
{A Paneitz-type operator for CR pluriharmonic functions,}
Bull. Inst. Math., Acad. Sin. {\bf 8} (2013), 285--322.

\bibitem[CY1]{CY1}S. Y. Cheng and S. T. Yau,
{\it On the regularity of the Monge--Amp\`{e}re equation $\det(\pa^2u/\pa x^i\pa x^j)=F(x, u)$,}
Comm. Pure Appl. Math., {\bf 30} (1977), 41--68.

\bibitem[CY2]{CY2}S. Y. Cheng and S. T. Yau,
{\it On the existence of a complete K\"ahler-Einstein metric on noncompact complex manifolds and the 
regularity of Fefferman's equation,}
Comm. Pure Appl. Math. {\bf 123} (1980), 507--544.

\bibitem[Fe]{Fe} C. Fefferman,  
{\it Monge--Amp\`ere equations, the Bergman kernel, and geometry of pseudoconvex domains,} 
Ann. of Math. {\bf 103} (1976), 395--416. 

\bibitem[FG]{FG}C. Fefferman and C. R. Graham,
{\it $Q$-curvature and Poincar\'{e} metric,} 
Math. Res. Lett. {\bf 10} (2003), 819--832.


\bibitem[Fo]{Fo}D. J. F. Fox,
{\it Geometric structures modeled on affine hypersurfaces and generalization of the Einstein Weyl
and affine hypersphere equations,}
{\tt arXiv:0909.1897}

\bibitem[G]{G}C. R. Graham, 
{\it Volume and area renormalizations for conformally compact Einstein metrics,}
Rend. Cirs. Mat. Palermo Ser. II {\bf 63} (2000), Suppl., 31--42.



\bibitem[GZ]{GZ}C. R. Graham and M. Zworski,
{\it Scattering matrix in conformal geometry,}
Invent. Math. {\bf 152} (2003), 89--118.

\bibitem[GMS]{GMS}C. Guillarmou, S.Moroianu, and J-M. Schlenker, 
{\it The renormalized volume and uniformisation of conformal structures,}
{\tt arXiv:1211:6705}

\bibitem[HS]{HS}M. Henningson and K. Skenderis, 
{\it The holographic Weyl anomaly,}
J. High Energy Phys. {\bf 7} (1998) 

\bibitem[HMM]{HMM}K. Hirachi, T. Marugame and Y. Matsumoto,
{\it Variations of total $Q$-prime curvature on CR manifolds,}
Adv. in Math. {\bf 306} (2017), 1333--1376.

\bibitem[H]{H}K. Hirachi,
{\it $Q$-prime curvature on CR manifolds,}
Differential Geom. Appl. {\bf 33} Suppl. (2014), 213--245. 

\bibitem[M]{M}T. Marugame,
{GJMS operators and $Q$-curvature for conformal Codazzi structures,}
 Differential Geom. Appl. {\bf 49} (2016), 176--196.

\bibitem[NS]{NS}K. Nomizu and T. Sasaki,
{\it Affine Differential Geometry: Geometry of Affine Immersions,}
Cambridge University Press, 1994.

\bibitem[R]{R}M. Randall,
{\it The conformal-to-Einstein equation on M\"obius surfaces,}
Differential Geom. Appl. {\bf 35} (2014), 274--290.

\bibitem[Sa1]{Sa1}T. Sasaki, 
{\it A note on characteristic functions and projectively invariant metrics on a bounded convex domain,}
Tokyo J. Math. {\bf 8} (1985), 49--79.

\bibitem[Sa2]{Sa2}T. Sasaki,
{\it On the characteristic function of a strictly convex domain and Fubini--Pick invariant,}
Results in Math. {\bf 13} (1988), 367--378.

\bibitem[Se]{Se}N. Seshadri,
{\it Volume renormalization for complete Einstein-K\"ahler metrics,}
Differential Geom. Appl. {\bf 25} (2007), 356--379.

\end{thebibliography}
\end{document}